\documentclass[journal]{IEEEtran}

\usepackage[utf8]{inputenc} 
\usepackage[T1]{fontenc}    
\usepackage{hyperref}       
\usepackage{url}            
\usepackage{booktabs}       
\usepackage{amsfonts}       
\usepackage{nicefrac}       
\usepackage{microtype}      
\usepackage{lipsum}
\usepackage{amsmath}
\usepackage{xcolor}
\usepackage{amssymb}
\usepackage{amsthm}
\usepackage{graphicx}
\usepackage{wrapfig,lipsum,booktabs}
\usepackage{float}

\usepackage{versions}
\includeversion{techreport}
\excludeversion{Ton-submit}

\usepackage[ruled]{algorithm2e} 

\newcommand{\x}{\mathbf{x}}
\newcommand{\e}{\mathbf{e}}

\newcommand{\q}{\mathbf{q}}
\newcommand{\ar}{\mathbf{a}}
\newcommand{\s}{\mathbf{s}}
\newcommand{\su}{_{ij}(t)}
\newcommand{\cali}{\mathcal{I}}
\newcommand{\calk}{\mathcal{K}}
\newcommand{\pr}{\mathbb{P}}
\newcommand{\ex}{\mathbb{E}}

\newcommand{\cblue}{\color{black}}

\newcommand{\cblack}{\color{black}}
\newcommand{\flip}{random $d$-flip }
\newcommand{\htlike}{heavy traffic behaviour is MaxWeight-like}

\newcommand{\norm}[1]{\left\lVert#1\right\rVert}

\newtheorem{theorem}{Theorem}

\newtheorem{lemma}[theorem]{Lemma}
\newtheorem{definition}[theorem]{Definition}
\newtheorem{asu}{Assumption}
\newtheorem{proposition}[theorem]{Proposition}

\DeclareMathOperator*{\argmax}{arg\,max}

\author{ Prakirt Raj Jhunjhunwala, Siva Theja Maguluri\\ prakirt@gatech.edu, siva.theja@gatech.edu \\ Georgia Institute of Technology}

\begin{document}
\title{Low-Complexity Switch Scheduling Algorithms: Delay Optimality in Heavy Traffic}

\maketitle
\thispagestyle{empty}
\pagestyle{empty}

\begin{abstract}
Motivated by applications in data center networks, in this paper, we study the problem of scheduling in an input queued switch. While throughput maximizing algorithms in a switch are well-understood, delay analysis was developed only recently. It was recently shown that the well-known MaxWeight algorithm achieves optimal scaling of mean queue lengths in steady state in the heavy-traffic regime, and is within a factor less than $2$ of a universal lower bound. However, MaxWeight is not used in practice because of its high time complexity. In this paper, we study several low complexity algorithms and show that their heavy-traffic performance is identical to that of MaxWeight. We first present a negative result that picking a random schedule does not have optimal heavy-traffic scaling of queue lengths even under uniform traffic. We then show that if one picks the best among two matchings or modifies a random matching even a little, using the so-called flip operation, it leads to MaxWeight like heavy-traffic performance under uniform traffic.  We then focus on the case of non-uniform traffic and show that a large class of low time complexity algorithms have the same heavy-traffic performance as MaxWeight, as long as it is ensured that a MaxWeight matching is picked often enough. We also briefly discuss the performance of these algorithms in the large scale heavy-traffic regime when the size of the switch increases simultaneously with the load. \cblue Finally, we perform empirical study on a new algorithm to compare its performance with some existing algorithms. \color{black} 

\end{abstract}

\section{INTRODUCTION}

Input queued crossbar switches are essential components in building networks and have been studied since the 90's \cite{mckeown1999achieving}. There is now renewed interest in studying input queued switches because they are good approximations of data center networks built using Clos topologies \cite{alizadeh2013pfabric}\cite{perry2014fastpass}. 

The throughput performance of various algorithms was studied in the past. It was shown in \cite{mckeown1999achieving}\cite{182479}  that the celebrated MaxWeight algorithm maximizes throughput. However, implementing a MaxWeight algorithm involves computing a maximum weight bipartite matching at every time, which has a complexity of $O(n^{2.5})$ \cite{duan2012scaling}, which is impractical given the size of today's data center networks. Therefore, lower complexity algorithms that also maximize throughput were studied in \cite{665071}\cite{1019350}\cite{giaccone2003randomized}\cite{shah2002efficient}\cite{mckeown1999islip}. A low complexity algorithm with distributed implementation is presented in \cite{gong2019qps}. 

\cblue While maximizing throughput is a first order metric and easy to study, the objective in a real world data center is to minimize delay. \color{black} 
Due to Little's law, studying steady-state delay is the same as studying steady-state mean queue length. However, evaluating either of these is challenging in queueing systems. Therefore, they are studied in various asymptotic regimes such as heavy-traffic. The primary focus of this paper is the heavy-traffic regime, where the switch is loaded close to its capacity. In this regime, the mean queue length goes to infinity, and we study the rate at which it goes to infinity by considering the sum of the queue lengths in heavy-traffic, multiplied by a heavy-traffic parameter ($\epsilon$) that captures the distance to the capacity region.

Heavy-traffic queue length behavior under MaxWeight was recently studied in \cite{doi:10.1287/15-SSY193}\cite{maguluri2016optimal}\cite{hurtadolange2019heavytraffic}  and an exact expression for the heavy-traffic scaled mean sum queue lengths was obtained. Moreover, it was shown that the queue lengths are within a factor of 2 from a universal lower bound, thus establishing that MaxWeight has an optimal scaling. Moreover, using Little's law, this result implies that the mean delay is $O(1)$ independent of the size of the switch. This result was obtained in \cite{doi:10.1287/15-SSY193} using a novel drift method. The key step is to establish a state space collapse (SSC) result, which shows that in heavy traffic, the $n^2$ dimensional queue length vector lives close to a $(2n-1)$ dimensional cone. The main challenge here was due to the multidimensional nature of the state space collapse. The goal of this paper is to study low complexity scheduling algorithms that have MaxWeight like queue length performance on heavy-traffic, i.e., within a constant factor of the universal lower bound. 

\begin{table*}[ht]
\label{tab:summary}
\centering
\caption{Results presented in this paper}
\begin{tabular}{|p{3.6cm}|p{1.7cm}|p{1.4cm}|p{2cm}|p{2.7cm}|p{2.9cm}|}
\hline
Algorithm                  & Throughput Optimality & $\displaystyle \lim_{\epsilon \downarrow 0}\epsilon \sum_{ij} \bar q_{ij}$ & $\ex[\sum_{ij}\bar q_{ij}] = O(n^{1+\beta})$ for  $\beta$ & Amortized Complexity     & \multicolumn{1}{l|}{Reference} \\ \hline
MaxWeight                  & Yes                   & $O(n)$                                                        & $>4$                                                                                                                                  & $O(n^{2.5})$                 & \cite{doi:10.1287/15-SSY193}        \\ \hline
Random                     & Uniform traffic       & $O(n^2)$                                        & N/A                                                                                                                                                & $O(n)$                   & Sec \ref{sec: random}         \\ \hline
Power-of-$d$               & Uniform traffic       & $O(n)$                                           & $>6$                                                                                                                                  & $O(dn)$                  & Sec \ref{sec: po2}, Sec \ref{sec: large_scale}        \\ \hline
Random $d$-Flip              & Uniform traffic       & $O(n)$                                           & $>6$                                                                                                                                  & $O(n+d)$                 & Sec \ref{sec: flip}, Sec \ref{sec: large_scale}         \\ \hline
Bursty MaxWeight           & Yes                   & $O(n)$                                                        & $>3+\max(\gamma,1)$                                                                                                                            & $O(n^{2.5}/m)$               & Sec \ref{sec: bursty}, Sec \ref{sec: large_scale}, \cite{1019350}      \\ \hline
Pipelined MaxWeight       & Yes                   & $O(n)$                                                        & $>3+\max(\gamma,1)$                                                                                                                            & $O(n^{2.5})$, parallelizable & Sec \ref{sec: pipelined}, Sec \ref{sec: large_scale}, \cite{1019350}        \\ \hline
Randomly Delayed MaxWeight & Yes                   & $O(n)$                                                        & Unknown                                                                                                                                                & $O(\delta n^{2.5})$        & Sec \ref{sec: randomly_delayed_maxweight}                          \\ \hline
Pick and Compare (PC-$d$)  & Yes                   & $O(n)$                                                        & Unknown                                                                                                                                                & $O(dn)$                  & Sec \ref{sec: randomly_delayed_maxweight}                                \\ \hline
LAURA                      & Yes                   & $O(n)$                                                        & Unknown                                                                                                                                                & $O(n\log^2n)$            &Sec \ref{sec: randomly_delayed_maxweight}.  \cite{giaccone2003randomized}                              \\ \hline
SERENA                     & Yes                   & $O(n)$                                                        & Unknown                                                                                                                                               & $O(n)$                   & Sec \ref{sec: randomly_delayed_maxweight}, \cite{giaccone2003randomized}        \\ \hline
$d$-Flip                     & Unknown                   & Unknown                                                        & Unknown                                                                                                                                               & $O(d)$                   & Sec \ref{sec: simulations}       \\ \hline
\end{tabular}
\end{table*}

\subsection{Main Contributions}
We first consider the switch under uniform traffic and study random scheduling, where a matching is picked every time uniformly at random. We show 
in Section \ref{sec: random}
that, under uniform traffic, while random scheduling achieves the maximum possible throughput, its heavy-traffic behavior is much worse. In particular, we show that for random scheduling, the heavy traffic scaled mean sum queue length is $\Theta(n^2)$, as opposed to $\Theta(n)$ for MaxWeight. This is because random scheduling does not exhibit state space collapse. 

Then, in Section \ref{sec: po2}, we
study the power-of-d scheduling, where $d$ matchings are picked uniformly at random and the best among them is used. We show that under uniform traffic, power-of-$d$ scheduling not only maximizes throughput, but also has MaxWeight like heavy-traffic behavior. Inspired from \cite{balinski1964primal}, we further propose an algorithm that we call \flip, where one matching is sampled at random, and one tries to improve it by trying to flip two queues in the matching. We show that under uniform traffic, this is enough to get maximum throughput and MaxWeight like heavy-traffic behavior. 

We then consider variants of the MaxWeight algorithm under general traffic
in Section \ref{sec:approx_maxweight}.
We show that the bursty MaxWeight algorithm and the pipelined MaxWeight algorithm \cite{1019350} have the same heavy-traffic performance as MaxWeight. In bursty MaxWeight, a maximum weight matching is computed every $m$ time steps, and the same matching is used for the $m$ steps. In pipelined MaxWeight, a maximum weight matching is computed at every time, but it takes $m$ time steps to complete this computation, and so the matching is used only $m$ steps later. This is amenable to a parallelized implementation. We present a general theorem that characterizes the heavy-traffic performance of a broad class of algorithms that includes both these algorithms. 

We then consider another general class of linear complexity algorithms proposed by Tassiulas \cite{665071} that are shown to be throughput optimal. In these algorithms, at any time, there is a small $\delta$ chance of picking a MaxWeight matching. If not, a matching is sampled according to some distribution, and it is compared with the previous matching, and the best among the two is used. This framework was used in \cite{giaccone2003randomized} and \cite{shah2002efficient} to develop several low complexity algorithms including APSARA, SERENA and LAURA. We show 
in Section \ref{sec: randomly_delayed_maxweight}
that this large class of algorithms also has the same heavy-traffic behavior as MaxWeight.

While all the algorithms that we study have the same heavy-traffic performance as that of the MaxWeight, they all are not equally good in practice. This is because while heavy-traffic analysis is finer than throughput optimality, it does not capture subtle differences in performance. In particular, any algorithm that exhibits SSC has MaxWeight like heavy-traffic performance. However, different algorithms may have slightly different quality of SSC. In order to capture this performance difference, we consider the large system heavy-traffic regime \cite{shah2011optimal}\cite{shah2014optimal}\cite{shah2016queue}\cite{xu2019improved}
in Section \ref{sec: large_scale}. 
In this regime, the size of the switch increases simultaneously while the traffic approaches the capacity, and we study the performance difference of the above algorithms in this regime.

All the results are summarized in Table \ref{tab:summary}. 
\cblue
In Section \ref{sec: simulations}, we use simulations to exhibit the performance of the proposed algorithm $d$-Flip. \color{black} We finally conclude in Section \ref{sec: conclusion}, along with a few pointers on future research directions. 
We will now start with the model, notation and other preliminaries such as a formal definition of state-space collapse and heavy traffic optimality in Section \ref{sec: model}.

\section{MODEL AND PRELIMINARIES}
\label{sec: model}
In this section, we present the model and introduce the required notation. \cblue Moreover, we present several known results from the previous literature. \color{black}
In any time slot $t$,  $q_{ij}(t)$ (also called queue length) denotes  the number of packets that needs to be transferred from the input $i$  to the output $j$,  $\mathbf{q}(t)$ is a $n\times n$ queue length matrix with elements  $q_{ij}(t)$. Throughout this paper, the letters in bold denotes vectors in $\mathbb R^{n \times n}$.
Also, for any process $x(t)$ that converges in distribution, $\bar x$ denotes the limiting random variable to which $x(t)$ converges.
\subsection{Arrival and Service Process}
\label{section: Arrival prcess}
At any time $t$, $a_{ij}(t)$ ($\mathbf{a}(t)$ in matrix form) denotes the number of packets that arrive at the input port $i$ to be delivered to output port $j$. \cblue The term matrix and vector are used interchangeably throughout the paper. \color{black} The mean arrival rate vector is denoted by $\mathbb{E}[\mathbf{a}(t) ] = \mathbf{\boldsymbol \lambda}$ and variance Var$(\mathbf{a}(t)) = \mathbf{\boldsymbol\sigma}^2$.
\begin{asu}
\label{assu: arrival_process}
For the arrival process:
\begin{itemize}
    \item[(i)] For any given pair $(i,j)$, $a_{ij}(t)$ are independent and identically distributed with respect to $t$.
    \item[(ii)] The arrival process is also independent across input-output pair, i.e., for all $i,j,i'$ and $j'$ such that $(i,j)\neq (i',j')$, $a_{ij}(t)$ is independent of $a_{i'j'}(t)$.
    \item[(iii)] There exists $ a_{\max} $ such that $ \forall i,j,t$, $a_{ij}(t) \leq a_{\max} < \infty$.
    \item[(iv)] There is non-zero probability of no arrivals, i.e., $\pr(\ar(t) = \mathbf{0}) > 0$, where $\mathbf 0$ is a $n\times n$ vector of all zeros.
\end{itemize}
\end{asu}
The assumptions mentioned in Assumption \ref{assu: arrival_process} are quite general for a switch system. 
Due to the structure of the switch system, in each time slot, each input can be matched with at most one output and vice-versa. The switch system can also be thought of as a complete bipartite graph with $2n$ nodes and $n^2$ edges. And the weight of each edge $(i,j)$ is $q_{ij}(t)$. A schedule is then a matching on the corresponding graph, which is represented by a $n\times n$  matrix with entries either $0$ or $1$. We use $\s(t)$ to denote the schedule in time slot $t$. The element $s_{ij}(t) = 1$ if and only if the input $i$ is connected with the output $j$ at time $t$. In this paper, without loss of generality, we consider a schedule to be a \cblue perfect matching between input and output nodes,  i.e., no more connections between input and output nodes can be made. \color{black} It follows that the set of possible schedules  $\mathcal{X}$ is just the set of all $n\times n$ permutation matrices.

The weight of the schedule is the sum of the queue lengths that are being served in the given time slot. A scheduling algorithm or policy picks the schedule $\mathbf{s}(t)$ in every time slot. \textit{MaxWeight} is a scheduling algorithm that always picks the schedule with the highest weight.
  If the algorithm picks schedules only from $\mathcal{X}$, it might happen that $s_{ij}(t)$ is $1$ but there are no packets available to be transferred from input $i$ to output $j$. In such a case, we say that the service is wasted. As a result, the queue length evolve according to the following equation,
\begin{align*}
q_{ij}(t+1) &= [q_{ij}(t) + a_{ij}(t) - s_{ij}(t)]^+\\
&= q_{ij}(t) + a_{ij}(t) - s_{ij}(t) + u_{ij}(t),
\end{align*}
where $[x^+] = \text{max}(0,x)$ and $u\su$ denotes the unused service on link $(i,j)$. By writing this into matrix form, we get
\begin{align*}
\mathbf{q}(t+1) &= \mathbf{q}(t) +\mathbf{a}(t) - \mathbf{s}(t) + \mathbf{u}(t)
\end{align*}
It can be observed that if $q_{ij}(t+1) >0$ then $u\su = 0$. This gives us the condition that, $q_{ij}(t+1) u\su = 0$
for all $(i,j)$ which implies that $\langle \q(t+1) ,\mathbf u (t) \rangle =0$.
\cblue
Let $\sigma(\mathcal H_t)$ be the $\sigma$-algebra generated by  $\mathcal H_t$, where $\mathcal H_t$ denotes the history till time $t$, i.e.,
\begin{equation}
\label{eq: filtration_history}
	\mathcal{H}_t = \{ \q(0),\s(0),\q(1),\dots, \s(t-1),\q(t) \}.
\end{equation}
Similarly, we define $\sigma(\tilde{\mathcal H_t}) $ to be the $\sigma$-algebra generated by  $ \tilde{\mathcal H_t}$, where 
\begin{equation}
\label{eq: filtration_history_2}
	 \tilde{\mathcal H_t} = \{ \q(0),\s(0),\q(1),\dots, \s(t-1),\q(t),\s(t) \}.
\end{equation}
\color{black}
\cblue
For an arbitrary scheduling algorithm, it not necessary that $\mathbf{q}(t)$ forms a Markov chain. For example, in Section \ref{sec: randomly_delayed_maxweight}, we look at the algorithm named as randomly delayed MaxWeight, where the system uses the MaxWeight schedule with probability $\delta$, and with probability $(1-\delta)$, it uses the schedule used in previous time slot. In such a case, the system need to remember the schedule used in previous time slot and so using $\mathbf{q}(t)$ as the state of Markov chain is not enough. The correct definition for the state of the Markov chain in this case would be $(\mathbf{q}(t),\mathbf{s}(t))$. 
\color{black}
For the switch system considered in this paper, we assume that there is process $X(t)$ such that $X(t)$ forms a Markov chain and we define two conditions on $X(t)$ as given below.
\cblue
	\begin{itemize}
		\item[A.1.] The Markov chain $X(t)$ is $\sigma(\tilde{\mathcal{H}_t})$-measurable and it is also irreducible and aperiodic. 
		\item[A.2.] There exists a function $g(\cdot)$ such that $\q(t) = g(X(t))$. Further, let $\mathcal{A} \subset \mathbb Z^{n\times n}$ and suppose  $g^{-1}(\mathcal A) = \{ X : g(X) \in \mathcal{A}\}$. Then, if $|\mathcal{A}| < \infty$ then $|g^{-1}(\mathcal{A})| < \infty$, where $|\cdot|$ denotes the size of the set. 
	\end{itemize}

The condition A.1 is required to use the Lyapunov's drift argument to establish the positive recurrence of the Markov chain $X(t)$. Note that irreducibility is not a major condition as otherwise, we can just consider the communicating class of $X(0)$ to be the state space. The condition A.2 essentially says that $\q(t)$ is a deterministic function of the state $X(t)$, which implies that the state of the Markov chain holds full information about the queue length, which is necessary for the technical analysis we are doing in this paper.  If such a function exists then the state $X(t)$ contains the information about $\q(t)$ within itself, which is necessary to define the Lyapunov functions considered in this paper. 
\color{black}

In this paper, we say that the switch system is \textit{stable} if the corresponding Markov chain $X(t)$ is positive recurrent. The capacity region $\mathcal{C}$ of the switch is the set of mean arrival rate vector $\mathbf{\boldsymbol \lambda}$ for which there exists some scheduling policy under which the switch system is stable. As given in \cite{182479}, the capacity region for a switch, denoted by $\mathcal{C}$ is
\cblue
\begin{equation*}
	\mathcal C = \Big \{ \mathbf{\boldsymbol \lambda} \in \mathbb{R}_+^{n \times n } : \sum_{i=1}^n \lambda_{ij} < 1,\sum_{j=1}^n \lambda_{ij} < 1 \ \forall i,j \Big \}.
\end{equation*}
\color{black}
 An algorithm for which the the queue length vector $\q(t)$ is stable for all $\mathbf{\boldsymbol \lambda} \in \mathcal{C}$ is called throughput optimal. In \cite{stolyar2004maxweight}, it was proved that MaxWeight is throughput optimal. 
 
 The set $\mathcal{F}$ denotes the set of doubly stochastic matrices. 
The set $\mathcal F$ forms a \cblue facet \cite[Chapter 3]{ziegler2012lectures} \cblack of the closure of the capacity region $\mathcal C$. Throughout this paper, we use $\mathbf{\boldsymbol \lambda}$ to denote a matrix in $\mathcal{C}$ and $\mathbf{\boldsymbol \nu}$ to denote a matrix in $\mathcal{F}$.

A switch system is in heavy traffic regime if the mean arrival rate matrix is very close to the boundary of the capacity region. Note that for any $\boldsymbol \lambda \in \mathcal{C}$, there exists $\boldsymbol \nu \in \mathcal{F}$ and $\epsilon_{ij} \in [0,1]$ such that $\lambda_{ij} = (1-\epsilon_{ij})\nu_{ij}$. In order to make the theoretical analysis simpler, we take $\epsilon_{ij} = \epsilon$ for all $(i,j)$. Otherwise we can pick an $\epsilon$ such that $\epsilon_{ij} \geq \epsilon \ \forall (i,j)$ and many of our upper bound results would still be valid. This is also called \textit{Completely Saturated Case} in \cite{doi:10.1287/15-SSY193}.
\begin{asu}
\label{assu: lambda}
The mean arrival rate vector is $\mathbf{\boldsymbol \lambda} = (1-\epsilon)\mathbf{\boldsymbol \nu}$, for some $\boldsymbol \nu \in \mathcal{F}$ and $\epsilon \in (0,1)$, such that
\begin{equation*}
   \nu_{\min} \triangleq \min_{ij} \nu_{ij}  >0. 
\end{equation*}
Also, \cblue there exists $ \tilde{\boldsymbol \sigma}^2$ such that \color{black} the variance $\mathbf{\boldsymbol \sigma}^2  \rightarrow \mathbf{\tilde{\boldsymbol \sigma}}^2$ as $\epsilon \downarrow 0$.
\end{asu}
The parameter $\epsilon$ in Assumption \ref{assu: lambda} is a measure of how far $\boldsymbol \lambda \in \mathcal{C}$ is from the boundary $\mathcal{F}$. In this paper, we refer $\epsilon$ as the heavy traffic parameter.
The switch system is in heavy traffic regime if $\epsilon$ is very close to $0$. 

From here onwards, we will assume that the arrival satisfies Assumption \ref{assu: arrival_process} and \ref{assu: lambda}. \cblue Throughout the paper, $\epsilon$ denotes the distance of $\boldsymbol \lambda$ from its corresponding $\boldsymbol \nu$ as given in Assumption \ref{assu: lambda}. Also, note that even though the parameters of the arrival process depends on $\epsilon$, we do not attach $\epsilon$ to their symbols just to keep the notations simple.
\color{black}


 An arrival process is said to be under \textit{uniform traffic} if the mean arrival rate for every input-output pair is same, i.e. $\lambda_{ij} = \lambda_{i'j'}$ for all $i,j,i'$ and $j'$. \cblue Also, even though the mean arrival rates are same, the variance might differ across the input-output node pairs. \color{black} It is easy to observe that for an arrival process that is in the capacity region and under uniform traffic, the mean arrival rate lies in $\mathcal{C}^*  \subset \mathcal{C}$ given by
 \cblue
 \begin{align*}
    \mathcal{C}^* &= \Big \{ \mathbf{\boldsymbol \lambda} \in \mathbb{R}_+^{n \times n } :  \lambda_{ij} < \frac{1}{n}, \ \forall i,j \Big \}.
\end{align*}
\color{black}

  Let  $\mathbf{1}$ be an $n \times n$ matrix of all ones. If the uniform traffic arrival process satisfies Assumption \ref{assu: lambda}, then we can take $\boldsymbol \lambda = \frac{1-\epsilon}{n}\mathbf{1}$. Furthermore, if the arrival process is \textit{uniform Bernoulli traffic}, i.e., the arrivals $a_{ij}(t)$ are Bernoulli random variables, then $\norm{\mathbf{\boldsymbol \sigma}}^2 = (1-\epsilon)(n-1 + \epsilon )$, which gives $\lim_{\epsilon \downarrow 0} \norm{\mathbf{\boldsymbol \sigma}}^2 = \norm{\mathbf{\tilde{\boldsymbol  \sigma}}}^2 = n-1$.

\subsection{Geometry}
Let $\mathbf{e}^{i}$ be an $n\times n$ matrix with $i^{th}$ row being all ones and zeros everywhere else and $\mathbf{\Tilde{e}}^j$ is a $n\times n$ matrix with  $j^{th}$ column begin all ones and zeros everywhere else. Consider the subspace $\mathcal{S} \subset \mathbb{R}^{n \times n}$ defined as,
\begin{equation*}
    \mathcal{S} = \Big \{ \mathbf{x}: \mathbf{x} = \sum_i w_i \mathbf{e}^i + \sum_j \Tilde{w}_j \mathbf{\Tilde{e}}^j \text{ s.t. } w_i , \Tilde{w}_j \in \mathbb{R} \ \forall i,j \Big \}.
\end{equation*}
We define the cone $\mathcal{K}$ to be the intersection of $\mathcal{S}$ with the positive orthant, i.e., $\mathcal{K} = \mathcal{S}\cap \mathbb{R}^{n\times n}_{+}$. The dimension of cone $\mathcal{K}$ is $2n-1$ as it is spanned by $2n-1$ independent vectors out of 2n vectors $\{\mathbf{e}^{i} \}$ and $\{\mathbf{\Tilde{e}}^j\}$. For two matrices $\mathbf{x}$ and $\mathbf{y}$ in $\mathbb{R}^{n\times n}$, $\langle \mathbf{x}, \mathbf{y} \rangle$ denotes the Frobenius inner product and $\norm{\mathbf{x}} = \sqrt{\langle \mathbf{x}, \mathbf{x} \rangle}$.

\cblue
For any vector $\mathbf x$, $\mathbf x_{\|}$ denotes the projection to the space $\mathcal{S}$ with $\mathbf{x}_{\perp} = \mathbf x - \mathbf x_{\|}$. Similarly, $\mathbf x_{\| \mathcal{K}}$ denotes the projection to the cone $\mathcal{K}$ with $\mathbf{x}_{\perp \mathcal{K}} = \mathbf x - \mathbf x_{\| \calk}$.
Some important properties regarding the set $\mathcal S$ and $\mathcal K$ are provided in Appendix A\begin{Ton-submit} (in the supplement file)\end{Ton-submit}.
\color{black}

\subsection{State-space collapse}
The main workhorse in heavy-traffic analysis is state-space collapse, viz., the phenomenon that a queueing system in heavy-traffic behaves like a system with a smaller number of queues.  
It was shown in \cite{doi:10.1287/15-SSY193} that in the switch system operating under MaxWeight scheduling algorithm, the state $\q(t)$ (of dimension $n^2$) collapses to the cone $\mathcal{K}$ (of dimension $2n-1$). This was established by showing that in steady state, $\q_{\perp \mathcal{K}}$ is significantly smaller than $\q_{\| \mathcal{K}}$. The following definition presents this notion of state space collapse more formally. 
\begin{definition}
\label{def: ssc}
	A scheduling algorithm is said to achieve \textit{State-Space Collapse} (SSC) if the switch system is stable, the corresponding Markov chain $X(t)$ satisfies condition A.1 and A.2 and there exists $\epsilon_0 >0$ such that for $0 < \epsilon \leq \epsilon_0$, the steady state queue length vector satisfies 
\begin{equation}
    \label{eq: qperp_bound}
    \mathbb{E} \Big [\norm{\Bar{\q}_{\perp \mathcal{K}}}^r \Big] \leq C_r \quad \forall r \in \{1,2,\dots\},
\end{equation}
where $C_r$ is a constant, independent of $\epsilon$.
\end{definition}

\begin{theorem}
\label{prop: ssc}
Consider a switch system which achieves state-space collapse according to Definition \ref{def: ssc}, then  the heavy traffic scaled queue length satisfies
\begin{equation}
\label{eq: heavy_traffic_sum}
     \lim_{\epsilon \downarrow 0} \epsilon \mathbb{E}  \Big [ \sum_{ij} \Bar{q}_{ij} \Big] = \Big (1- \frac{1}{2n} \Big)\norm{\mathbf{\tilde{\boldsymbol  \sigma}}}^2.
\end{equation}
 \end{theorem}

 The proof of the result in Eq. \eqref{eq: heavy_traffic_sum} for MaxWeight was given in \cite{doi:10.1287/15-SSY193}. However, the proof in \cite[Theorem 1]{doi:10.1287/15-SSY193} implies that Eq. \eqref{eq: heavy_traffic_sum} holds for any scheduling algorithm that satisfies SSC as given by Definition \ref{def: ssc}. Now, we formally define the term \textit{MaxWeight-like}.
 \cblue
 \begin{definition}
\label{def: maxweight_like}
	For a switch scheduling algorithm, its heavy traffic performance is said to be \textit{MaxWeight-like} if the algorithm satisfies Eq. \eqref{eq: heavy_traffic_sum}.
\end{definition}
According to Theorem \ref{prop: ssc}, to show that an algorithm is \textit{MaxWeight-like} it is enough to prove that the algorithm achieves the SSC according to Definition \ref{def: ssc}.
\color{black}
 Although MaxWeight satisfies Eq. \eqref{eq: heavy_traffic_sum}, there might exist algorithms that perform better than MaxWeight in heavy traffic. As mentioned in \cite{doi:10.1287/15-SSY193}, we only know that for any scheduling algorithm,
\cblue
\begin{equation}
\label{eq: universal_lower_bound_no_limit}
	\epsilon  \mathbb{E}  \Big [ \sum_{ij} \Bar{q}_{ij} \Big] \geq   \frac{1}{2} \norm{\mathbf{\boldsymbol \sigma}}^2 - \frac{\epsilon (1- \epsilon)}{2} \stackrel{\epsilon \downarrow 0}{ \longrightarrow} \frac{1}{2 } \norm{\mathbf{\tilde{\boldsymbol  \sigma}}}^2.
\end{equation}
\color{black}
This means that the heavy traffic scaled mean sum queue length for MaxWeight is within a factor of 2 of the optimal. In \cite{10.1145/3199524.3199563}, the authors presented an algorithm which performs better than MaxWeight, although they did not provide the heavy traffic limit for it. 


\subsection{Lyapunov Drift}
We use Lyapunov drift arguments to obtain the heavy-traffic results in this paper. To that end, in this subsection, we present some Lyapunov functions, their drift and some known results on Lyapunov drift.

Let $X(t)$ be an irreducible and aperiodic Markov chain over a countable state space $\mathcal A$. Suppose $Z : \mathcal{A} \rightarrow \mathbb{R}_{+}$ is a non-negative Lyapunov function. The drift of $ Z$ at $X$ is the change in the value of $Z(\cdot)$ after one step transition. Mathematically, 
\begin{equation*}
    \Delta Z(X) \triangleq \big[Z(X (t+1)) - Z(X(t)) \big] \cali (X(t) = X ),
\end{equation*}
where $\cali(\cdot)$ is the indicator function. We define three different conditions on the drift:
\begin{itemize}
 	\item[ C.1.] There exists $\eta >0$ and $\kappa < \infty$ such that $\forall t>0$ and $\forall X \in \mathcal{A}$ with $Z(X) \geq \kappa$, 
    \begin{equation*}
        \mathbb{E} \big[ \Delta Z(X) \big| X(t) = X \big] \leq -\eta.
    \end{equation*}
    \item [C.2.] There exists $D < \infty$ such that $\forall X \in \mathcal{A}$,
    \begin{equation*}
       \pr \big( |\Delta Z(X)| \leq D \big) = 1. 
    \end{equation*}
 	\item [C.3.] There exists a non-negative random variable $M$ such that $|\Delta Z(X)|$ is stochastically dominated by $M$ for all $t\geq 0$, i.e., for any $c>0$,
	\begin{equation*}
		\mathbb{P} \big(|\Delta Z(X)| > c \big| X(t) = X \big) \leq \mathbb{P}(M > c) \ \ \forall t \geq 0,
	\end{equation*}
     and $\mathbb{E} [e^{\theta M}] < \infty$ for some $\theta >0$.
\end{itemize}

It is easy to observe that the condition C.2 is stronger than condition C.3. We define C.2 and C.3 differently because we can state a stronger result if the condition C.2 holds. \cblue Some important results related to the drift analysis of switch system is given in Appendix B\begin{Ton-submit} (in the supplement file)\end{Ton-submit}. \cblack

\section{Class 1: Modifications of Random Scheduling} 
\label{sec:ssc_under_uniform_traffic}

In this section, we study  random scheduling and some modifications of it. 
For a switch system, random scheduling is not throughput optimal. The capacity region of random scheduling is known to be $\mathcal{C}^*$. Throughout this section, we  assume that the arrival process is under uniform traffic. We show that heavy traffic behaviour of random scheduling is not MaxWeight-like, but there are some variants of random scheduling which have MaxWeight-like heavy traffic behaviour.



\subsection{Random Scheduling}
\label{sec: random}

Random scheduling, as the name suggests, is a scheduling policy for which the schedule $\s(t)$ is chosen uniformly at random from the set of permutation matrices $\mathcal X$. The time complexity of generating a random schedule is $O(n)$ by using Fisher–Yates shuffle \cite[Example 12]{fisher1938statistical}.

\begin{proposition}
\label{prop: stable_random}
Consider a switch system under uniform traffic. For random scheduling, the process $\q(t)$ forms a positive recurrent Markov chain and,
\begin{equation}
\label{eq: random_uniform}
     \lim_{\epsilon \downarrow 0} \epsilon \mathbb{E}  \Big [ \sum_{ij} \Bar{q}_{ij} \Big] = \frac{n}{2}\norm{\mathbf{\tilde{\boldsymbol  \sigma}}}^2 + \frac{n(n-1)}{2}.
\end{equation} 
Moreover, if the arrival process is uniform Bernoulli traffic,
\begin{equation}
\label{eq: random_bernoulli}
     \lim_{\epsilon \downarrow 0} \epsilon \mathbb{E}  \Big [ \sum_{ij} \Bar{q}_{ij} \Big] = n(n-1).
\end{equation} 
\end{proposition}

From Theorem \ref{prop: ssc}, we know that any scheduling algorithm that satisfies SSC has  optimal queue length scaling of $O(\norm{\tilde{\boldsymbol  \sigma}}^2)$ in heavy traffic. While from Proposition \ref{prop: stable_random}, the heavy traffic scaled mean sum queue length for random scheduling is $O(n\norm{\tilde{ \boldsymbol \sigma}}^2)$. This shows that random scheduling does not have optimal queue length scaling.

\cblue
Proposition \ref{prop: stable_random} under uniform Bernoulli traffic was presented in \cite[Theorem 2]{1019350}, and was proved by noting that under random scheduling, each of the $n^2$ queues of the switch can be treated as independent single server queues. The proof for general traffic can be shown similarly, and we present the details in Appendix C\begin{Ton-submit} (in the supplement file)\end{Ton-submit} for completeness. 
\color{black}


\subsection{State space collapse }
\label{sec: class_1}

In this section, we will present the proof of the heavy-traffic results for a class of scheduling algorithms. \cblue Later on we provide some examples that lie in this class like power-of-d, and \flip scheduling algorithms that are modification of random scheduling. \color{black} 
\begin{definition}
	\label{Def:class_1}
	A scheduling algorithm lies in class $\Pi_1(\boldsymbol \nu)$ if the corresponding Markov chain $X(t)$ satisfy condition A.1 and A.2 and there exists a constant $W_1>0$ such that in any time slot $t\geq 0$, the expected weight satisfies
	\begin{equation}
	\label{eq: class1}
		\mathbb{E}[ \langle \q(t),\s(t) \rangle | X(t) = X ] \geq \langle \q,\boldsymbol \nu \rangle + W_1 \norm{\q_{\perp \mathcal{K}}},
	\end{equation}
	where $W_1$ is independent of $\epsilon$ and $\q =g(X)$.
\end{definition}
MaxWeight lies in class $\Pi_1(\boldsymbol \nu)$ for any $\boldsymbol \nu$ for which $\nu_{\min} >0$, in which case $W_1=\nu_{\min}$ \cite{doi:10.1287/15-SSY193}. 
We later on show that, power-of-$d$ and \flip scheduling lies in $\Pi_1 \big(\frac{1}{n} \mathbf 1 \big)$ with $W_1 = \frac{1}{2n^3}$. Also, it is easy to observe that random scheduling does not lie in class $\Pi_1 \big(\frac{1}{n} \mathbf 1 \big)$. Next, we claim that any scheduling algorithm that lies in class $\Pi_1 (\boldsymbol \nu)$ satisfies SSC if the mean arrival rate is $\boldsymbol \lambda = (1-\epsilon )\boldsymbol \nu$.

\begin{theorem}
	\label{thm: ssc_class_1}
Suppose the mean arrival rate is of the form $\boldsymbol \lambda = (1-\epsilon )\boldsymbol \nu$ and the scheduling algorithm lies in the class $\Pi_1(\boldsymbol \nu)$. Then,  the scheduling algorithm achieves SSC and so its \htlike.
\end{theorem}

 
 \cblue
 The proof of Theorem \ref{thm: ssc_class_1} follows by showing that if we pick the Lyapunov function to be $\norm{\q_{\perp \calk}} $, then this Lyapunov function satisfy the conditions C.1 and C.2. After that we can use existing results to show that all the moments of $\norm{\q_{\perp \calk}} $ are bounded by a constant. This implies that scheduling algorithm achieves SSC according to Definition \ref{def: ssc} and so it is \htlike. The details of the proof are provided in Appendix D\begin{Ton-submit} (in the supplement file)\end{Ton-submit}. Next, we provide some examples of the algorithms that lie in this class.
 \color{black}

\subsection{Power-of-d scheduling}
\label{sec: po2}

The \textit{power-of-$d$} scheduling is a variant of random scheduling in which the system samples $d \geq 2$ schedules uniformly at random with replacement from the set of permutation matrices $\mathcal X$ and chooses the one with the largest weight. In time slot $t$, let $\{\s_1(t),\dots ,\s_{d}(t)\}$ denotes the schedules sampled by the power-of-$d$ algorithm. The schedule chosen by power-of-$d$ is
\begin{equation*}
    \s(t) = \argmax \{\langle \q(t),\s_1(t) \rangle, \dots, \langle \q(t),\s_d(t) \rangle  \}.
\end{equation*}
We assume that schedules are sampled with replacement just for simplicity. The results does not change qualitatively even if the schedules are sampled without replacement. Generating a random schedule has a time-complexity of $O(n)$. And as power-of-$d$ generated $d$ random schedule times, the time complexity of power-of-$d$ is $O(dn)$.

It is known that power-of-$d$ scheduling is not throughput optimal \cite{giaccone2003randomized}. However, it is stable under all the arrival rates in $\mathcal{C}^*$, and so we can study its heavy traffic behavior under uniform traffic.
\subsection{Random d-Flip scheduling}
\label{sec: flip}
Random $d$-flip scheduling algorithm is another variant of random scheduling. For any given schedule $\s_1$ and a queue length matrix $\q$, a \textit{flip step} constitutes of following three steps,
\begin{itemize}
	\item Sample two indices $(i,j)$ and $(k,l)$ uniformly at random such that $s_{1,ij} = s_{1,kl}=1$, \cblue where $s_{1,ij}$ is the $(i,j)^{th}$ element of the schedule $\mathbf{s}_1$. \color{black}
	\item Create a different schedule $\s_2$ such that $s_{2,ij} = s_{2,kl}= 0$ and $s_{2,il} = s_{2,kj}=1$.
	\item Select the schedule with the larger weight, i.e.,
		\begin{equation*}
			\s = \argmax_{\s_1,\s_2} \{\langle \q,\s_1 \rangle, \langle \q,\s_2 \rangle  \}.
 		\end{equation*}
\end{itemize}
Note that to compare the weight of the matching $\s_1$ and $\s_2$ in the flip step, the system does not need to calculate the weight of the schedule. It suffices to compare the value of $q_{ij}+ q_{kl}$ and $q_{il}+ q_{kj}$. Thus the flip step has a complexity of only $O(1)$. 

In each time slot $t$, random $d$-flip samples a schedule $\s(t)$ uniformly at random from the set $\mathcal X$ and then uses the flip step on $\s(t)$, $d$ times consecutively.  As the complexity of generating a random schedule is $O(n)$ and complexity of flip step is $O(1)$, the complexity of random $d$-flip is $O(n+d)$.

The flip step considered in this paper is random flipping and it is not necessary that flip step improves the schedule generated by random sampling, but there are more ways to implement the flip step. In \cite{balinski1964primal}, authors provide another method of implementing the flip step, which strictly improves the weight of the schedule but the complexity of each flip step is $O(n)$. The algorithm APSARA in \cite{giaccone2003randomized} is also based on flip step mentioned above.

\begin{lemma}
\label{lemma: expectation_bound}
In any time slot $t$, the schedule chosen by power-of-$d$ or by random $d$-flip satisfies,
\begin{equation}
\label{eq: expectationbound}
    \mathbb{E} \big[ \langle \q(t),\s(t) \rangle \big| \q(t) = \q \big] \geq \frac{1}{n} \langle \q,\mathbf 1 \rangle + \frac{1}{2 n^3} \norm{\q_{\perp \mathcal{K}}}.
\end{equation}
\end{lemma}

\cblue
The proof of Lemma \ref{lemma: expectation_bound} uses some clever manipulation of the expected weight of the schedule in each time slot. Note that for random scheduling, the expected weight is $\frac{1}{n} \langle \q,\mathbf 1 \rangle$. 

Power-of-$d$ generates more schedules to improve the weight. We show that in expectation, this improvement is at least $\frac{1}{2 n^3} \norm{\q_{\perp \mathcal{K}}}$. The detailed proof of Lemma \ref{lemma: expectation_bound} for power-of-$d$ is provided in Appendix E\begin{Ton-submit} (in the supplement file)\end{Ton-submit}. 

Similarly, the algorithm random $d$-flip first samples a random schedule and then
implements the flip steps that strictly improves the expected weight. We prove that the expected improvement by the first flip step is at least $\frac{1}{2 n^3} \norm{\q_{\perp \mathcal{K}}}$. The details of the proof of Lemma \ref{lemma: expectation_bound} for random $d$-flip is provided in Appendix F\begin{Ton-submit} (in the supplement file)\end{Ton-submit}.

\color{black}

\begin{proposition}
	\label{prop: po2_ssc}
	Under uniform traffic, power-of-$d$ scheduling and random $d$-flip achieve SSC and so their \htlike.
\end{proposition}

\begin{proof}
For power-of-$d$ and random $d$-flip, the process $\q(t)$ forms a Markov chain and satisfy condition A.1 and A.2. \cblue The aperiodicity of Markov chain $ \q(t)$ in this case follows from part (iv) of Assumption \ref{assu: arrival_process}, as the state $\q = \mathbf 0$ has a self loop. And irreducibility follows by taking the state space to be the set of states reachable from $\mathbf 0$ as given in \cite[Exercise 4.2]{srikant2014communication}. Also, it is evident that the chain $\q(t)$ satisfies the condition A.2. 

From Lemma \ref{lemma: expectation_bound}, we know that power-of-$d$ and random $d$-flip lies in the class $\Pi(\frac{1}{n} \mathbf{1} )$ as given in Definition \ref{Def:class_1}. Then, Proposition \ref{prop: po2_ssc} follows directly from Theorem \ref{thm: ssc_class_1}.
\color{black}
\end{proof}

So far, we considered uniform traffic since power-of-$d$ and \flip scheduling algorithms are not throughput optimal. The switch is unstable under general non-uniform traffic under these algorithms. One way to overcome this limitation is by using a load-balanced switch  \cite{Chang:2002:LBB:2284907.2285091}.
A load balanced switch is a two-stage architecture consisting of two switches in tandem. The first stage aims to equalize the arrival rate across the inputs of the switch at the second stage, so that the second stage is operating under uniform traffic. The on-line complexity of operating the first switch is just $O(1)$ so it does not affect the overall performance. More details regarding the load balanced setup can be found in \cite{srikant2014communication} and 
\cite{Chang:2002:LBB:2284907.2285091}. 




\section{Class 2: Approximate MaxWeight} 
\label{sec:approx_maxweight}

In this section, we will present another class of scheduling policies that achieves SSC and so are heavy traffic optimal. In \cite{1019350}, the authors present two efficient approximations of the MaxWeight \cblue named bursty MaxWeight and pipelined MaxWeight. Next, we define a class of algorithm that contains these two algorithms, and provide the heavy result for that class. \color{black}

\subsection{State space collapse}
\label{sec: class2_ssc}
Now we prove the heavy traffic result for a class of algorithms which includes bursty MaxWeight and pipelined MaxWeight. 
\begin{definition}
	\label{Def:class_2}
	A scheduling algorithm lies in class $\Pi_2$ if the corresponding Markov chain $X(t)$ satisfy condition A.1 and A.2 and there exists a constant $W_2\geq 0$ such that in any time slot $t\geq 0$, the expected weight satisfies
	\begin{equation}
	\label{eq: class_2}
		\mathbb{E} \big[ \langle \q(t),\s(t) \rangle \big| X(t) = X \big] \geq \max_{\s} \langle \q,\s \rangle - W_2,
	\end{equation}
	where $W$ is independent of $\epsilon $ and $\q = g(X)$. 
\end{definition}

The class $\Pi_2$ presented in Definition \ref{Def:class_2} is based on the class of algorithms presented in \cite{1019350}. MaxWeight lies in class $\Pi_2$ with $W_2 = 0$. In \cite{1019350}, it was proved that any scheduling algorithm in the class $\Pi_2$ is throughput optimal. Next, we look at the SSC and heavy traffic optimality of scheduling algorithms in class $\Pi_2$. 

\begin{theorem}
\label{thm: ssc_class_2}
 Any scheduling algorithm that lies in the class $\Pi_2$ achieves SSC and so its \htlike. 
 \end{theorem}

 Theorem \ref{thm: ssc_class_2} shows that bursty MaxWeight and pipelined MaxWeight satisfies SSC and thus their \htlike. 
\cblue
The proof of Theorem \ref{thm: ssc_class_2} follows on similar lines as the proof of Theorem \ref{thm: ssc_class_1}. From the definition of the algorithms in class $\Pi_2 $, the weight of schedule for any scheduling algorithm in $\Pi_2$ is at most a constant difference away from MaxWeight. 
If the queue lengths are very large (like in heavy traffic), the weight is of the MaxWeight schedule is much larger compared to difference $W_2$ and so the performance of the scheduling algorithm is quite close to that of MaxWeight. 
Thus, the heavy traffic performance of algorithms in $\Pi_2$ is similar to MaxWeight scheduling. The proof of Theorem \ref{thm: ssc_class_2} is provided in Appendix G\begin{Ton-submit} (in the supplement file)\end{Ton-submit}.
\color{black}

\subsection{Bursty MaxWeight}
\label{sec: bursty}
This scheduling algorithm evaluates the MaxWeight schedule after every $m$ time-slots and then uses the same schedule consecutively for next $m$ time slots. 
For a $n\times n$ switch, the time-complexity of computing the MaxWeight schedule is $O(n^{2.5})$. 
Thus, the amortized time-complexity of bursty MaxWeight is $O(n^{2.5}/m)$.
Note that if $m$ is chosen to be $\Theta(n^{2.5})$, this leads to a constant amortized complexity. 


\subsection{Pipelined MaxWeight}
\label{sec: pipelined}
	This takes $m$ time slots to compute the MaxWeight schedule, so the MaxWeight schedule corresponding to $\q(t)$ is used in time slot $t+m$. While pipelined MaxWeight still has a high complexity of $O(n^{2.5})$, it is amenable to a parallelized implementation which makes it useful in practice.
\begin{proposition}
	\label{prop: piplelined_ssc}
 Bursty MaxWeight and pipelined MaxWeight achieve SSC and so their \htlike.
\end{proposition}
\begin{proof}
	The proof of condition A.1 and A.2 for both algorithms are given in Appendix H\begin{Ton-submit} (in the supplement file)\end{Ton-submit}. For both algorithms, as shown in \cite{1019350},
		\begin{equation*}
			\langle \q(t),\s(t) \rangle \geq \max_{\s} \langle \q(t),\s \rangle - 2mna_{\max}.
		\end{equation*}
Thus, bursty and pipelined MaxWeight lies in the class $\Pi_2$ as given in Definition \ref{Def:class_2} (Section \ref{sec: class2_ssc}) and then by using Theorem \ref{thm: ssc_class_2}, both algorithms satisfy SSC and so their \htlike.
\end{proof}

\cblue

Both bursty MaxWeight and pipelined MaxWeight depend on the parameter $m$. Even though the result in Proposition \ref{prop: piplelined_ssc} holds for any value of $m$, it does not mean that the heavy traffic performance of bursty MaxWeight or pipelined MaxWeight is not affected by the value of $m$. The larger the value of $m$, the further away these algorithms are from MaxWeight. Later, in Section \ref{sec: large_scale}, we provide an intuition of the effect of $m$ on the heavy traffic behavior of the switch. 
\color{black}

\section{\textcolor{black}{Class 3: Randomized Algorithms with Memory}}
\label{sec: randomly_delayed_maxweight}

In this section, we look at the third class of algorithms that satisfies SSC. The description of the class is as follows.

\begin{definition}
	\label{Def: class_3}
	A scheduling algorithm lies in class $\Pi_3$ if the corresponding Markov chain $X(t)$ satisfy condition A.1 and A.2 and
			\begin{itemize}
			\item[(i)] There exists a $\delta>0$ such that for every time $t\geq 0$ the chosen schedule $\s(t)$ satisfies
			\begin{equation}
			\label{eq: delta}
				\mathbb{P}\big( \langle \q(t) ,\s(t) \rangle = \max_{\s} \langle \q(t) ,\s \rangle \big| \mathcal{H}_t\big) \geq \delta,
			\end{equation}
			where $\mathcal H_t$ is given by Eq. \eqref{eq: filtration_history}.
			\item[(ii)] For every time $t\geq 1$, the chosen schedule $\s(t)$ satisfies
			\begin{equation}
			\label{eq: comparison}
				\langle \q(t), \s(t) \rangle \geq \langle \q(t), \s(t-1) \rangle. 
			\end{equation}
			\item[(iii)] There exists a deterministic function $f(\cdot)$, such that $(\q(t),\s(t)) = f(X(t))$.
		\end{itemize}
\end{definition}

It is easy to observe that MaxWeight scheduling lies in $\Pi_3$ with $\delta = 1$. The definition of class $\Pi_3$ in this paper is based on the class of algorithms presented in \cite{665071}. The algorithms presented in \cite{665071} uses a two-step procedure to choose the schedule $\s(t)$.
\begin{itemize}
    \item \textit{Sampling step:} The system samples a schedule $\tilde \s (t)$ such that which satisfies Eq. \eqref{eq: delta}.
    \item \textit{Comparison step:} The sampled schedule $ \tilde \s(t)$ is compared with $\s(t-1)$, i.e.,
\begin{equation*}
	\s(t) = \argmax_{\tilde \s(t) , \s(t-1)} \big \{ \langle \q(t), \tilde \s(t)  \rangle , \langle \q(t), \s(t-1)  \rangle \big \}.
\end{equation*}
\end{itemize}
Any scheduling algorithm that uses the above mentioned steps satisfies Eq. \eqref{eq: delta} and Eq. \eqref{eq: comparison}. Note that the complexity of the comparison step is $O(n)$, so the comparison step does not affect the complexity of the algorithm with worse than linear time complexity. The comparison step is very useful because it plays a key role in making the scheduling algorithm throughput optimal. Some of the algorithms based on the procedure given in \cite{665071} are as follows,
\begin{itemize}
	\item \textit{Randomly Delayed MaxWeight:} This is randomized version of bursty MaxWeight. The system chooses to implement MaxWeight with probability $\delta $ or uses the previous schedule with probability $1-\delta$. The amortized complexity of this algorithm is $O(\delta n^{2.5})$. 
	\item \textit{Pick and Compare (PC-d):} This algorithm is an extension of power-of-$d$. In this algorithms, the system generates the random schedule $\tilde \s(t)$ during the sampling step using power-of-$d$ and then uses the comparison step. In this case, $\delta$ can be taken to be $d/n!$ and the complexity of pick and compare or PC-$d$ is $O(dn)$. 
	\item \textit{LAURA and SERENA:} In \cite{shah2002efficient} and \cite{giaccone2003randomized}, the authors presented several low-complexity algorithms, like LAURA (complexity $O(n\log^2 n)$) and SERENA (complexity $O(n)$), that also lie in the class $\Pi_3$ under the assumption that the arrival process is Bernoulli.
\end{itemize}

 The corresponding Markov chain for algorithms mentioned above and the class of algorithms in \cite{665071} is given by $X(t) = (\q(t),\s(t))$, and it can be observed that $X(t)$ satisfy condition A.1 and A.2. Thus, the class of algorithms in \cite{665071} also lies in class $\Pi_3$. By the arguments presented in \cite{665071}, for any scheduling algorithm that lies in class $\Pi_3$, the queue length process $\q(t)$ is stable, so we skip the proof of stability here.

\begin{theorem}
\label{thm:tassiulas_ssc}
Suppose the scheduling algorithm lies in the class $\Pi_3$. Then, the process $\q(t)$ is stable. Also, the scheduling algorithm achieves SSC and its \htlike.
\end{theorem}

Theorem \ref{thm:tassiulas_ssc} shows that the class of scheduling algorithms presented in \cite{665071} achieves SSC and have the same heavy traffic scaled queue length as MaxWeight. Therefore, the same result holds for the algorithms presented in \cite{shah2002efficient}. 

Let $\{T_k\}_{k\geq 0}$ be the sequence time instants at which the chosen schedule matches with the MaxWeight schedule, i.e.,
\begin{equation*}
	\langle \q(T_k), \s(T_k) \rangle = \max_{\s} \langle \q(T_k), \s \rangle. 
\end{equation*}
As $(\q(t), \s(t)) = f(X(t))$, it follows that $\{T_k\}_{k\geq 0}$ form a sequence of stopping times for the Markov chain $X(t)$. We define another process $\{Y_k\}_{k\geq 0}$ such that $Y_k = X(T_k)$. As $\{T_k\}_{k\geq 0}$ are stopping times, by strong Markov property, $Y_k$ forms a Markov chain. Also, for $Y_k = Y$, take $(\q,\s) = f(Y)$, and by the construction of  $\{Y_k\}_{k\geq 0}$, $\s$ is the MaxWeight schedule corresponding to $\q$. 

\begin{lemma}
\label{prop:tau_dominate}
	 For any $k\geq 1$, let $\tau_k = T_{k+1} - T_{k}$. Then, for any $Y$, the random variable $\{\tau_k|Y_k = Y\}$ is stochastically dominated by a random variable $M$ which is Geometrically distributed with mean $1/ \delta$.
\end{lemma}
The proof of Lemma \ref{prop:tau_dominate} follows from part (i) of Definition \ref{Def: class_3}. The idea is that in any time slot, there is at least $\delta$ probability that the MaxWeight schedule is picked. So, we can bound the probability $\pr( \tau_k > c | Y_k = Y)$ with $\pr(M>c)$. The proof of  Lemma \ref{prop:tau_dominate} is provided in Appendix I\begin{Ton-submit} (in the supplement file)\end{Ton-submit}.
 Next, we provide the proof sketch for Theorem \ref{thm:tassiulas_ssc}.
 
 \cblue

    The idea behind the proof of Theorem \ref{thm:tassiulas_ssc} is that the scheduling algorithms lying in $\Pi_3$ chooses the MaxWeight schedule frequently. In fact, Lemma \ref{prop:tau_dominate} shows that the time difference between choosing the two MaxWeight schedules is stochastically dominated by a geometrically distributed random variable. 
    
    Once the MaxWeight schedule is choosen, the weight of the schedule choosen by the scheduling algorithm in subsequent time slots is not much worse than the weight of the MaxWeight schedule in those time slot. This happens because the queue lengths cannot deviate too much in a single time slot (as arrivals are bounded) and the algorithm tries to improve upon the schedule used in previous time slot. 
    
    Mathematically, we use the the Lyapunov drift argument on the Markov chain $\{Y_k\}_{k\geq 0}$. We show that for $\{Y_k\}_{k\geq 0}$, if we choose the Lyapunov function to be $\norm{\q_{\perp \calk}}$, then this Lyapunov function satisfies the condition C.1 and C.3. The detailed proof of Theorem \ref{thm:tassiulas_ssc} is provided in Appendix J\begin{Ton-submit} (in the supplement file)\end{Ton-submit}.
\color{black}

\cblue
\section{Comparison of algorithms}

So far, we studied the performance of three classes of algorithms, and saw that they all have MaxWeight like heavy-traffic performance. In this section, we present a comparison and contrast them and Table \ref{tab:summary} presents a summary. 
The three classes are clearly not disjoint. For instance, as mentioned before, MaxWeight lies in all three classes. It is not hard to construct other algorithms that lie in all three classes. one can combine multiple scheduling algorithms to generate a new algorithm. 
For example, one can create a scheduling algorithm that generate two schedules, one by bursty MaxWeight and other by power-of-$d$ and then chooses the one with larger weight. Such an algorithm will lie in both class $\Pi_1(\frac{1}{n} \mathbf{1})$ and $\Pi_2$. Another example is PC-$d$ (Section \ref{sec: randomly_delayed_maxweight}), which was created by implementing an extra comparison step after doing power-of-$d$. So PC-$d$ lie in both $\Pi_{1}(\frac{1}{n} \mathbf{1})$ and $\Pi_3$. A similar modification  can be done with random $d$-flip without changing the complexity of the algorithm. 

While we proved that they all have MaxWeight-like heavy-traffic mean delay performance, their performance under other metrics can be different. 
First consider throughput optimality. As mentioned before, while 
the algorithms in class $\Pi_2$ and class $\Pi_3$ are throughput optimal, 
the algorithms in $\Pi_1(\boldsymbol \nu)$ are not. 
Power-of-$d$ and random $d$-flip are known to be stable only in a subset of the capacity region $\mathcal{C}^*$. 
In particular, they support maximum possible load only when the traffic is uniform. There are known examples \cite[Theorem 1]{giaccone2003randomized} showing how they are unstable under non-uniform load for certain arrival rate vectors within the capacity region. 
In the next subsection, we present the  large scale heavy traffic regime, which is yet another asymptotic performance view, that enables us to distinguish between the performance of the algorithms presented so far. 

\subsection{Large scale heavy traffic regime}
\label{sec: large_scale}
\cblue
Stochastic networks such as an input queued switch are in general hard to analyze and so, are usually studied in various asymptotic regimes with heavy-traffic being a prominent one that is the main focus of this paper. In heavy-traffic regime, we fix the size of the switch $n$, and load it to its maximum capacity, i.e., we let the heavy-traffic parameter $\epsilon \rightarrow 0$. Another popular regime is the large scale limit, where the load is fixed ($\epsilon$ is fixed) and the size of the system, $n$ is sent to infinity \cite{neely2007logarithmic}. Different regimes present different view points of the system, and obtaining results in various regimes presents a more holistic view. For example, several algorithms that have the same performance in one regime may have different performance in another regime. 

In this section, we consider a whole spectrum of asymptotic regimes between the large scale regime and the heavy-traffic regime, where the size of the switch simultaneously grows to infinity as the arrival rate approaches the boundary of the capacity region. These are called the \textit{large scale heavy traffic regime}. These regimes are of special interest today, since the size of today's data center networks is huge.  
\color{black}
For a given $n$, let the $\epsilon(n)$ denote the heavy traffic parameter of the system such that $\epsilon(n)$ is $\Omega(n^{-\beta})$ for some $\beta >0$.
Note that the large-scale regime corresponds to the case when  $\beta =0$, and the heavy-traffic regime can be thought of as the case when $\beta = \infty$.
Also, for simplicity, we assume that the arrival process is uniform Bernoulli traffic in this section. This means that $\boldsymbol \lambda = \frac{(1-\epsilon(n))}{n} \mathbf 1$ and $$\norm{\mathbf{\boldsymbol \sigma}}^2 = (1-\epsilon(n))(n-1 + \epsilon(n) ) = n-1 + o(n).$$ 

\cblue
In this case, from the universal lower bound given in Eq. \eqref{eq: universal_lower_bound_no_limit}, we know that if $\epsilon(n) = \Omega(n^{-\beta})$ then, $\mathbb{E}  \big [ \sum_{ij} \Bar{q}_{ij} \big] $ is $\Omega(n^{1+\beta})$ for any value of $\beta>0$. Therefore, a natural question is if there is an algorithm under which,  we can also obtain an upper bound that is $\Omega(n^{1+\beta})$ for all $\beta>0$. At this point, while this is still an open question \cite{xu2019improved}, it is known  \cite[Corollary 1]{doi:10.1287/15-SSY193} that under MaxWeight algorithm, $\mathbb{E}  \big [ \sum_{ij} \Bar{q}_{ij} \big] $ is $\Omega(n^{1+\beta})$ for  $\beta> 4$. More precisely, for $\beta> 4$, under MaxWeight algorithm,
\begin{equation}
		\label{eq: beta_result}
		\lim_{n\rightarrow \infty} \frac{\epsilon(n)}{n} \mathbb{E}  \Big [ \sum_{ij} \Bar{q}_{ij} \Big] = 1.
	\end{equation}

The following theorem states similar results for the other algorithms studied so far. 
\begin{theorem}
\label{prop: beta_algos}
    Consider a switch system under uniform Bernoulli traffic such that $\lambda = \frac{(1-\epsilon(n))}{n} \mathbf 1$, and $\epsilon(n)$ is $\Omega(n^{-\beta})$. Then, we have Eq. \eqref{eq: beta_result}
     \begin{enumerate}
        \item[(i)] under algorithms in class $\Pi_1(\frac{1}{n} \mathbf{1})$ for $\beta > 3+\alpha_1$ where $W_1 = O(n^{-\alpha_1})$. In particular, power-of-$d$ and random $d$-flip algorithms satisfy Eq. \eqref{eq: beta_result} for $\beta>6$.
        \item[(ii)] under algorithms in class $\Pi_2$ for $\beta > 2 + \max\{\alpha_2, 2\}$ where $W_2 = O(n^{\alpha_2})$. In particular, bursty MaxWeight and pipelined MaxWeight algorithms satisfy Eq. \eqref{eq: beta_result}  for $\beta>3 + \max\{\gamma,1\}$, where $m$ is $O(n^\gamma)$.
    \end{enumerate}
\end{theorem}

The proof of the theorem is presented in Appendix K\begin{Ton-submit} (in the supplement file)\end{Ton-submit}.
The main tool to prove results of the form Eq. \eqref{eq: beta_result} in general, and the above theorem in particular is the following result that exploits a finer handle on the state space collapse, which follows from  \cite[Corollary 1]{doi:10.1287/15-SSY193}. The 
details are also provided in Appendix L\begin{Ton-submit} (in the supplement file)\end{Ton-submit}.



\begin{lemma}
\label{thm: beta_result}
	Consider a switch system under uniform Bernoulli traffic such that $\lambda = \frac{(1-\epsilon(n))}{n} \mathbf 1$, and $\epsilon(n)$ is $\Omega(n^{-\beta})$. Suppose the scheduling algorithm satisfy state space collapse with,
	\begin{equation*}
		\mathbb{E} \Big [\norm{\Bar{\q}_{\perp \mathcal{K}}}^r \Big] \leq C_r \quad \forall r \in \{1,2,\dots\},
	\end{equation*}
	such that $C_r$ is $O(n^{\alpha r})$ for all $r \in \{1,2,\dots\}$, then Eq. \eqref{eq: beta_result} holds for all $\beta > \alpha +1$.
\end{lemma}
\cblue

Note that for classical heavy-traffic results, we just need existence of SSC as in Definition \ref{def: ssc} that only cares about the existence of parameter $C_r$, and not about their dependence on the system size $n$. In contrast, here the quality of SSC, i.e., the exact dependence of $C_r$ on the system size $n$ plays a key role. The quality of such SSC is in turn influenced by the drift with which the algorithm pushes towards the cone. 
Among the algorithms studied so far, MaxWeight has the strongest drift towards the cone, and so we have from prior work \cite{doi:10.1287/15-SSY193} that Eq. \eqref{eq: beta_result} is valid for $\beta>4$. The class of algorithms studied in this paper have weaker drift towards the cone. While this distinction was not evident in the performance in the classical heavy-traffic regime, it becomes clearer in Theorem \ref{prop: beta_algos}.

More precisely, for Class 1 algorithms, the quality of drift towards the cone is determined by the constant $W_1$ in Eq. \eqref{eq: class1}. For MaxWeight, which also lies in Class 1, under uniform traffic, $W_1 = \nu_{\min} = 1/n$. For power-of-$d$ and and random $d$-flip  on the other hand, from Lemma \ref{lemma: expectation_bound}, we have that  $W_1 = \frac{1}{2n^3}$, which is worse by a factor of $O(1/n^2)$. This leads to  $\beta >6$ in Theorem \ref{prop: beta_algos} for these algorithms, which is 2 more than $\beta>4$ for MaxWeight. 

For Class 2 algorithms, the quality of drift towards the cone is determined by the constant $W_2$ in Eq. \eqref{eq: class_2}. For bursty MaxWeight and pipelined MaxWeight in this class, according to Definition \ref{Def:class_2} we have that $W_2 = 2nma_{\max}$ which is same as $O(n^{1+\gamma})$ when $m$ is $O(n^\gamma)$. On the other hand, for MaxWeight (which also lies in Class 2), $W_2=0$. It turns out that distinction leads to $\beta >3 + \max\{1,\gamma\}$ for pipelined and bursty MaxWeight algorithms. The details of the proof of Theorem \ref{prop: beta_algos} is provided in Appendix K\begin{Ton-submit} (in the supplement file)\end{Ton-submit}.

At this point, it is not clear if a result similar to Theorem \ref{prop: beta_algos} can be proved for Class 3 algorithms and that is an open question for further investigation. The key challenge is that the parameter $C_r$ in Lemma \ref{thm: beta_result} depends on  $1/\delta$, and  can be as large as $n!$. 
Another open question is investigating these algorithms for the values of $\beta$ not covered in Theorem \ref{prop: beta_algos}. It is known \cite{xu2019improved} to be a challenging open problem to even study the MaxWeight algorithm when $\beta \leq4$.

	
\color{black}



\cblue
\section{\texorpdfstring{$d$}{d}-flip: Empirical results}
\label{sec: simulations}

In this section, we present simulation results on an $O(d)$ algorithm named as $d$-flip. The algorithm $d$-flip differs from the random $d$-flip in the sense that $d$-flip do not generate a random schedule, it just uses the flip step $d$ times on the schedule used in previous time slot, i.e. it generates $\s(t)$ by applying $d$ flip steps on $\s(t-1)$. We only present empirical results related to $d$-Flip, as we cannot claim that  $d$-flip lies in any of the three class mentioned in this paper. The complexity of $d$-flip is $O(d)$ as it uses just $d$ flip step, each of complexity $O(1)$.
For simplicity, we will use the term \textit{Q-length} to denote $\ex \big[ \sum_{ij} \bar q_{ij} \big]$.  Also, the term \textit{Load} denotes the value $(1-\epsilon)$, where $\epsilon$ is the heavy traffic parameter. 

Fig. 1 shows the effect for increasing the value of $d$ in power-of-$d$, random $d$-flip and $d$-flip. The plot shows that increasing the $d$ does not have a huge effect on power-of-$d$, but it has significant effect on random $d$-flip and $d$-flip. For small values of $d$, power-of-$d$ perform better, while for large values of $d$, random $d$-flip and $d$-flip perform better. Intuitively,  the reason behind this is that for smaller values of $d$, power-of-$d$ has a higher probability of choosing a schedule with large weight as compared to random $d$-flip or $d$-flip, and this changes as the value of $d$ increases. For example, consider a schedule which can be converted to the MaxWeight schedule by a flip step. In this case, random $d$-flip or $d$-flip have $2/n(n-1)$ probability of choosing a flip step that gives MaxWeight schedule, while for power-of-$d$, the probability of sampling a MaxWeight schedule is $1/n!$.

\begin{figure}
\label{Fig: ton1}
\hspace{0.25em} \includegraphics[scale=0.55]{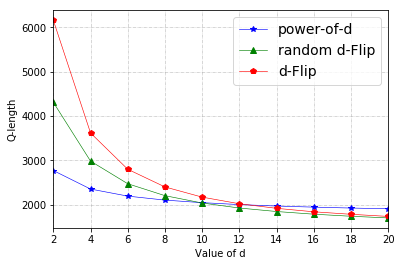}
	\caption{\textit{Q-length} vs \textit{Value of d} plot for power-of-$d$, random $d$-flip and $d$-flip for a $16 \times 16$ switch under uniform Bernoulli traffic with $Load =0.90$.}
\end{figure}

In Fig. 2, we again show the comparison of PC-$d$ and $d$-flip. In this plot, the arrival process is non-uniform. We know that PC-$d$ lies in the class $\Pi_3$, so it is throughput optimal and also heavy traffic optimal. Even though  $d$-flip does not lie in any of the classes mentioned in this paper, we can see that $d$-flip heavily outperforms PC-$d$.

\begin{figure}
	\label{Fig: ton7}
	\includegraphics[scale=0.55]{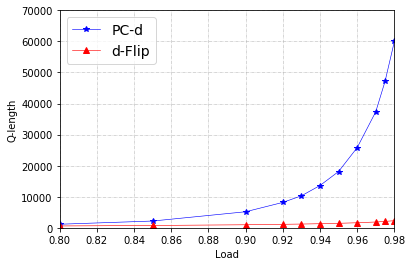}
	\caption{\textit{Q-length} vs \textit{Load} plot for PC-$d$ and $d$-flip with $d=8$ for a $16 \times 16$ switch under non-uniform Bernoulli traffic.}	
\end{figure}
\color{black}

\section{Future work}
\label{sec: conclusion}
In this section, we present a few future directions and open problems. One open problem is characterizing the exact stability region of power-of-$d$ or random $d$-flip. In this paper, we only looked at these algorithms under uniform traffic, or when the mean arrival rate lies in $\mathcal C^*$. But the stability region of these algorithms is larger than $\mathcal C^*$. Once the capacity region is understood, one can then study these algorithms under nonuniform traffic as long as the load is within their capacity region. 


While this paper studies a three different classes of low complexity algorithms, there are a few more algorithms that do not fall in any of the classes, and so are not analytically understood. The $d$-flip is one such algorithm, which is seen to perform well in simulations presented in Section \ref{sec: simulations}. Another example is iSLIP \cite{mckeown1999islip}, which commonly used in data centers, but the heavy traffic result for iSLIP is not known. 

Another future direction is the large scale analysis of algorithms in Class $\Pi_3$. Such an analysis will help us further differentiate between the algorithms in Class $\Pi_3$. Since simulations from Section \ref{sec: simulations} indicate that some algorithms such as PC-$d$, LAURA and SERENA from the class  $\Pi_3$ perform well, one expects that for these algorithms, a large scale heavy traffic regime result might be true.

\bibliographystyle{IEEEtran}
\bibliography{sample-base.bib}


\begin{IEEEbiography}[{\includegraphics[width=1in,height=1.25in,clip,keepaspectratio]{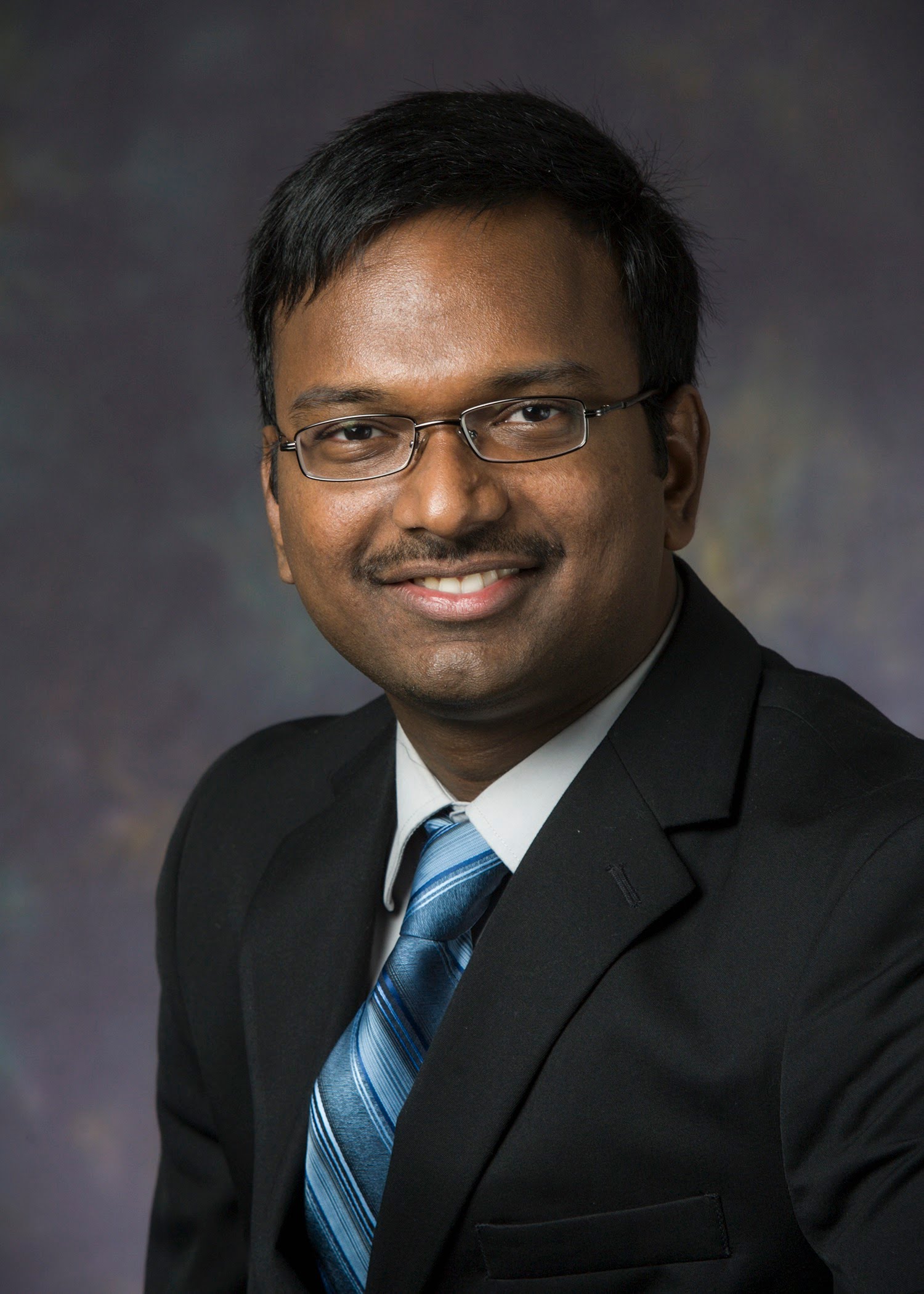}}]{Siva Theja Maguluri}
  is Fouts Family Early Career Professor and Assistant Professor in the School of Industrial and Systems Engineering at Georgia Tech. He obtained his Ph.D. and MS in ECE as well as MS in Applied Math from UIUC, and B.Tech in Electrical Engineering from IIT Madras. His research interests span the areas of Networks, Control, Optimization, Algorithms, Applied Probability and Reinforcement Learning. He is a recipient of the biennial “Best Publication in Applied Probability” award in 2017, “CTL/BP Junior Faculty Teaching Excellence Award” in 2020 and “Student Recognition of Excellence in Teaching: Class of 1934 CIOS Award” in 2020.
\end{IEEEbiography}

\begin{IEEEbiography}[{\includegraphics[width=1in,height=1.25in,clip,keepaspectratio]{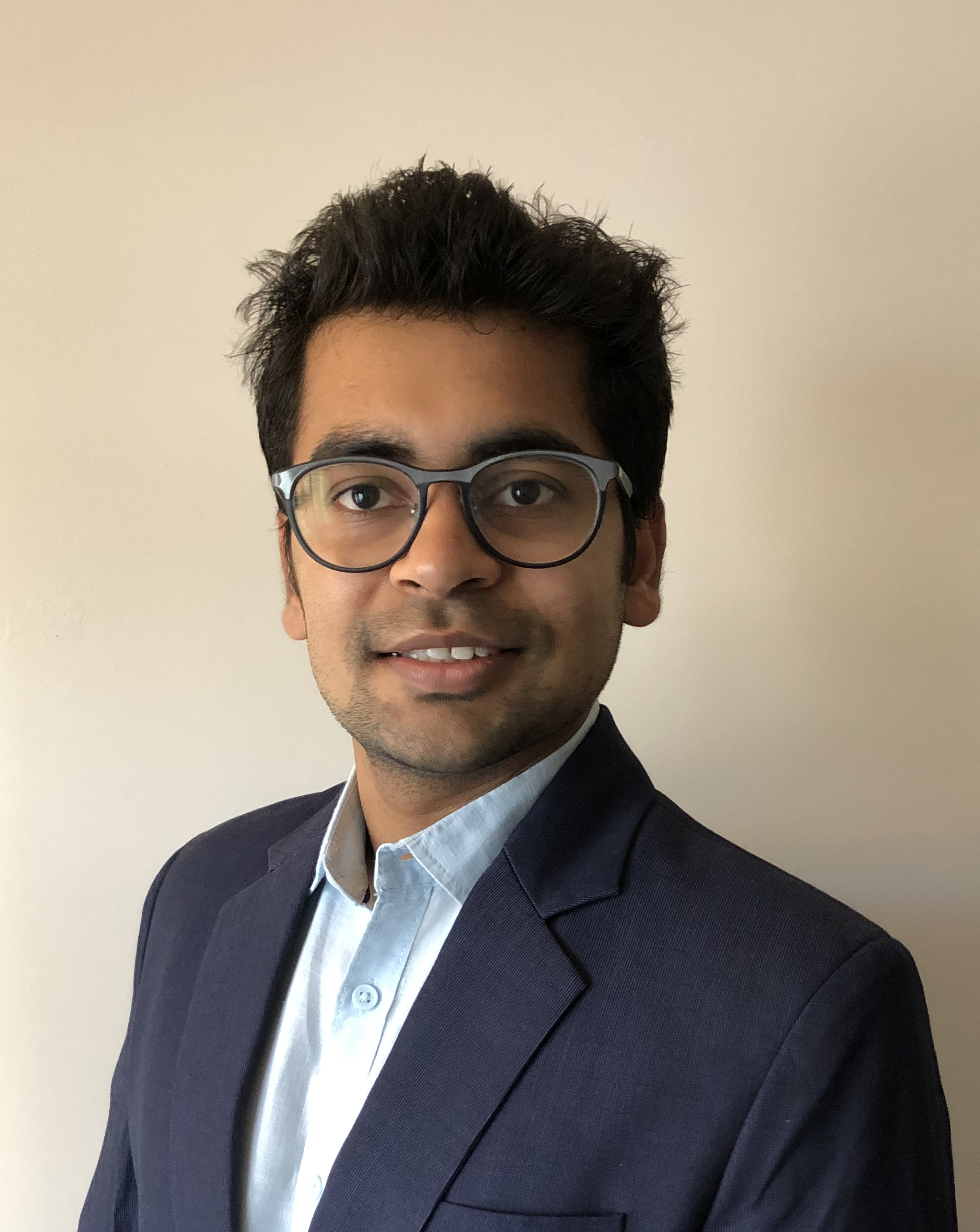}}]{Prakirt Raj Jhunjhunwala} 
   is a Ph.D. student with major in Operations Research and minor in Mathematics at ISyE, Georgia Institute of Technology. He obtained his B.Tech in Electrical Engineering from IIT Bombay. His research interests are Data Center Networks, Queueing Theory, Stochastic Processing Networks, Reinforcement Learning and Simulation Optimization. He is the Recipient of "Best Paper Award" in SPCOM 2018.
\end{IEEEbiography}

\pagebreak

\begin{techreport}
\appendix

\subsection{Appendix A}
\begin{lemma}
\label{lemma: projection_space}
	Let $\mathbf{x} \in \mathbb{R}^{n \times n}$, $\x_{\|}$ denotes its projection onto $\mathcal{S}$ and $\mathbf{x}_{\perp} = \mathbf{x} - \mathbf{x}_{\|}$. Then,
	\begin{itemize}
		\item[(i)] The closed form expression of $\x_{\|}$ is given by,
			\begin{equation*}
 		   		x_{\| ij} = \frac{1}{n} \sum_{j'=1}^n x_{ij'}+\frac{1}{n} \sum_{i'=1}^n x_{i'j} - \frac{1}{n^2} \sum_{i'=1}^n \sum_{j'=1}^n x_{i'j'}.
			\end{equation*}
		\item[(ii)]  For all $i$ and $j$, $ \langle \x_{\|} , \mathbf{e}^i \rangle =  \langle \x , \mathbf{e}^i \rangle$ and $\langle \x_{\|} , \mathbf{\Tilde{e}}^j \rangle =  \langle \x ,\mathbf{\Tilde{e}}^j \rangle$. And $\langle \x_{\perp} , \mathbf{e}^i \rangle = \langle \x_{\perp} , \mathbf{\Tilde{e}}^j \rangle =0$.
		\item[(iii)] For any $\boldsymbol \nu \in \mathcal F$, $\langle \x_{\|} , \boldsymbol \nu \rangle = \frac{1}{n} \sum_{ij}  x_{ij} = \frac{1}{n} \langle \x, \mathbf{1} \rangle$, where $\mathbf 1$ is a $n \times n$ matrix of all ones.
	\end{itemize}
\end{lemma}

Part (i) of Lemma \ref{lemma: projection_space} is provided in Appendix A in \cite{doi:10.1287/15-SSY193}, and part (ii) and (iii) follows directly from part (i). 

\begin{lemma}
\label{lemma: projection_cone}
	 Let $\mathbf{x} \in \mathbb{R}^{n \times n}$, $\x_{\| \mathcal{K}}$ denotes the projection onto $\mathcal{K}$ and $\mathbf{x}_{\perp \mathcal{K}} = \mathbf{x} - \mathbf{x}_{\| \mathcal{K}}$. Then
	 \begin{itemize}
	 	\item[(i)] $\x_{\| \mathcal{K}}$ and $\mathbf{x}_{\perp \mathcal{K}}$ are orthogonal, i.e.,  $\langle \mathbf{x}_{\perp \mathcal{K} } ,\x_{\| \mathcal{K}} \rangle =0$.
	 	\item[(ii)] For any $\mathbf x, \mathbf y \in \mathbb{R}^{n \times n}$, $\langle \mathbf x_{\| \calk}, \mathbf y_{\perp \calk} \rangle \leq 0$.
	 	\item[(iii)] For any $\boldsymbol \nu,\boldsymbol \nu' \in \mathcal F$, $\langle \x_{\| \calk} , \boldsymbol \nu \rangle = \langle \x_{\| \calk} ,\boldsymbol  \nu' \rangle$. 
	 	\item[(iv)] For any $\q \in \mathbb{R}^{n\times n}$ and $\boldsymbol \nu \in \mathcal{F}$,
	 	\begin{equation*}
\label{eq: nu_min}
   \mathbf{\boldsymbol \nu} + \frac{\nu_{\min}}{\norm{\q_{\perp \mathcal{K}}}} \q_{\perp \mathcal{K}} \in \mathcal{C}.
\end{equation*}
	 \end{itemize}
\end{lemma}
Details regarding Lemma \ref{lemma: projection_space} and \ref{lemma: projection_cone} can be found in \cite[Chapter E]{dattorro2010convex}. The proof of Lemma \ref{lemma: projection_cone} part (iv) is given in Claim 2 in \cite{doi:10.1287/15-SSY193}. We also provide the proof of Lemma \ref{lemma: projection_cone} part (iv) below.

\begin{proof}[Proof of Lemma \ref{lemma: projection_space} part (iv)]
As $|q_{\perp \mathcal{K} ij}| \leq  \norm{\q_{\perp \mathcal{K}}}$, we get that $\nu_{ij} + \frac{\nu_{\min}}{\norm{\q_{\perp \mathcal{K}}}} q_{\perp \mathcal{K} ij} \geq 0$ and so $\mathbf{\boldsymbol \nu} + \frac{\nu_{\min}}{\norm{\q_{\perp \mathcal{K}}}} \q_{\perp \mathcal{K}} \in \mathbb{R}^{n\times n}_+$. Also, as we know that $\mathbf{e}^i \in \mathcal{K }$, so $\langle \q_{\perp \mathcal{K}}, \mathbf{e}^i \rangle \leq 0$ by Lemma \ref{lemma: projection_cone} part (ii). Thus, we have that $\forall i$,
\begin{align*}
    \Big \langle \mathbf{\boldsymbol\nu} + \frac{\nu_{\min}}{\norm{\q_{\perp \mathcal{K}}}} \q_{\perp \mathcal{K}}, \mathbf{e}^i \Big \rangle &= \langle \mathbf{\boldsymbol\nu} , \mathbf{e}^i \rangle + \frac{\nu_{\min}}{\norm{\q_{\perp \mathcal{K}}}} \langle \mathbf{q}_{\perp \mathcal{K}} , \mathbf{e}^i \rangle \\
    & \leq \langle \mathbf{\boldsymbol\nu} , \mathbf{e}^i \rangle \\
    & = 1,
\end{align*}
where the last equality holds because $\mathbf{\boldsymbol\nu}\in \mathcal{F}$. Note that same arguments holds for $\Tilde{\mathbf{e}}^j$ for all $j$. This completes the proof.
\end{proof}

\subsection{Appendix B}
\begin{lemma}
	\label{lemma: drift_results}
	Consider a switch system and let the corresponding Markov chain $X(t)$ and let $\q(t) = g(X(t))$. Consider the Lyapunov functions defined as follows,
\begin{align*}
   V(X) \triangleq \norm{\q}^2 && W_{\perp \mathcal{K}}(X) \triangleq \norm{\q_{\perp \mathcal{K}}} && V_{\| \mathcal{K}}(X) \triangleq \norm{\q_{\| \mathcal{K}}}^2,
\end{align*} 
	 where $\q = g(X)$. Then, 
	\begin{itemize}
		\item[(i)] Let $K = \norm{\mathbf{\boldsymbol \lambda}}^2 + \norm{\mathbf{\boldsymbol \sigma}}^2 + n$, then,
			\begin{align*}
    			\mathbb{E} \big[  \Delta V(X) \big| X(t)=X \big] \leq & K  + 2(1-\epsilon)\langle \q ,\boldsymbol  \nu \rangle \\
   				 								& - 2 \mathbb{E}[ \langle \q(t),\s(t) \rangle | X(t) = X ].
			\end{align*}
		\item[(ii)] The drift $\Delta V_{\| \calk} (X)$ satisfies 
		\begin{equation*}
			\ex \big[ \Delta V_{\| \calk} (X)\big| X(t)=X \big] \geq -2\epsilon \langle \q_{\| \calk},\boldsymbol  \nu \rangle.
		\end{equation*}
		\item[(iii)] The drift $\Delta W_{\perp \mathcal{K}}(X)$ satisfies condition C.2 with $D = na_{\max}$, i.e.,
		\begin{equation*}
			\pr \big( |\Delta W_{\perp \mathcal{K}}(X)| \leq n a_{\max} \big) =1.
		\end{equation*}
		\item [(iv)] The drift $\Delta W_{\perp \mathcal{K}}(X)$ is related to the drift $\Delta V(X)$ and  $\Delta V_{\| \mathcal{K}} (X)$ by following equation,
			\begin{equation*}
				\label{eq: perp_drift}
    			\Delta W_{\perp \mathcal{K}}(X) \leq \frac{1}{2\norm{\q_{\perp \mathcal{K}}}} ( \Delta V(X) - \Delta V_{\| \mathcal{K}}(X)).
			\end{equation*}
	\end{itemize}
\end{lemma}
The proof of every part in Lemma \ref{lemma: drift_results} is given in \cite{doi:10.1287/15-SSY193}, so we skip the proof here. Next, we state the Lemma that is of key importance in order to prove the state-space collapse for the switch system. 

\begin{lemma}
\label{lemma: moment_bound}
Consider an irreducible and aperiodic Markov chain $X(t)$ over a countable state space $\mathcal{A}$ and suppose $Z : \mathcal{A} \rightarrow \mathbb{R}_{+}$ is a non-negative Lyapunov function. Further assume that Markov chain $X(t)$ converges in distribution to $\overline{X}$.
\begin{itemize}
	\item[(i)] If the drift $ \Delta Z(X)$ follows condition C.1 and C.2, then, for any $r = 1,2,\dots$,
	\begin{equation}
	\label{eq: part1_moment_bound}
    	\mathbb{E} \big[Z(\overline{X})^r \big] \leq (2\kappa)^r + (4D)^r \Bigg ( \frac{D + \eta}{\eta} \Bigg )^{r} r!,
	\end{equation}
	where $\kappa, \eta$ and $D$ are same as that defined in C.1 and C.2. 
	\item[(ii)] If the drift $ \Delta Z(X)$ follows condition C.1 and C.3, there exists a constant $C>0$ such that,
		\begin{equation*}
			\mathbb{E} \big[e^{\theta Z(\overline{X})} \big] \leq C,
		\end{equation*}
	where $\theta$ is same as that in condition C.2. This implies that for any $r = 1,2,\dots$, there exists $C_r <\infty$ such that $\mathbb{E} \big[Z(\overline{X})^r \big] < C_r$.
\end{itemize}

\end{lemma}
The proof of part (i) of Lemma \ref{lemma: moment_bound} follows from \cite[Lemma 2 and Lemma 3]{doi:10.1287/15-SSY193} and the proof of part (ii) of Lemma \ref{lemma: moment_bound} follows from \cite{hajek1982hitting}.

\subsection{Appendix C}
\begin{proof}[Proof of Proposition \ref{prop: stable_random}]

For random scheduling, the chosen schedule does not depend on the history of the system. Thus, the process $\q(t)$ itself forms a Markov chain. So, we take $X(t) = \q(t)$. The aperiodicity of Markov chain $ \q(t)$ in that case follows from part (iv) of Assumption \ref{assu: arrival_process}, as the state $\q = \mathbf 0$ has a self loop. And irreducibility follows by taking the state space to be the set of states reachable from $\mathbf 0$ as given in \cite[Exercise 4.2]{srikant2014communication}. Also, it is evident that the chain $\q(t)$ satisfies the condition A.2. 

For random scheduling, each of the $n^2$ queues behave independently and the service process for each queue is Bernoulli i.e., for any $(i,j)$, $s_{ij}(t)$ are Bernoulli random variables with $\mathbb{E}[ s_{ij}(t) ] = \frac{1}{n}$. Under the uniform traffic, the mean arrival rate for each queue is $\lambda_{ij} = \frac{(1-\epsilon)}{n}$. This implies that the mean service rate is higher than the mean arrival rate. So, the Markov chain $\q(t)$ is positive recurrent and the switch system is stable.

Also, as the service process for each queue is Bernoulli with mean $\frac{1}{n}$, $\text{Var}(s_{ij}(t)) = \frac{n-1}{n^2}$. Then by using the Kingman's formula \cite{srikant2014communication}, we get that for each $(i,j)$,
\begin{equation*}
    \lim_{\epsilon \downarrow 0} \epsilon \mathbb{E}[\Bar{q}_{ij}] = \frac{n}{2} \tilde \sigma^2_{ij} + \frac{n(n-1)}{2}.
\end{equation*}
By adding the above results for all $(i,j)$, we get the result given in Eq. \eqref{eq: random_uniform}. Finally, the result in Eq. \eqref{eq: random_bernoulli} follows directly from Eq. \eqref{eq: random_uniform} by substituting $\norm{\tilde{\boldsymbol  \sigma}}^2 = n-1$ for  uniform Bernoulli traffic.
\end{proof}



\subsection{Appendix D}

\begin{proof}[Proof of Theorem \ref{thm: ssc_class_1}]
In order to prove SSC, we first need to prove that for any scheduling algorithm that lies in class $\Pi_1(\boldsymbol \nu)$, the switch system is stable. Recall \cblue that by Definition \ref{Def:class_1}, there exists a Markov chain $X(t)$ that  satisfy condition A.1 and A.2 \cblack and there exists $g(\cdot)$ such that $\q(t) = g(X(t))$ for all $t\geq 0$. Take $\q = g(X)$. \cblue By Definition \ref{Def:class_1}, 
\begin{equation*}
    \mathbb{E}[ \langle \q(t),\s(t) \rangle | X(t) = X ] \geq \langle \q,\boldsymbol \nu \rangle + W_1 \norm{\q_{\perp \mathcal{K}}}.
\end{equation*}
By substituting this in Lemma \ref{lemma: drift_results} part (i), for the Lyapunov function $V(X)$ defined in Lemma \ref{lemma: drift_results}, \color{black}
\begin{align*}
	\mathbb{E} \big[  \Delta V(X) \big| X(t)=X \big] &\leq  K  - 2\epsilon\langle \q , \boldsymbol \nu \rangle-  2W_1\norm{\q_{\perp \calk}} \\
	&\leq  K  - 2\epsilon\langle \q ,\boldsymbol  \nu \rangle\\\
	& \leq K - 2 \epsilon \nu_{\min} \langle \q , \mathbf 1 \rangle \\
	& \stackrel{(a)}{\leq} - \epsilon \nu_{\min} \langle \q , \mathbf 1 \rangle ,
\end{align*}
where $(a)$ holds whenever $\langle \q , \mathbf 1 \rangle \geq  K / \epsilon \nu_{\min}$. It is easy to observe that the set of $\q$ for which $\langle \q , \mathbf 1 \rangle <  K / \epsilon \nu_{\min}$ forms a finite set whenever $\nu_{\min}>0$. By condition A.2, the set of $X$ for which $\langle \q , \mathbf 1 \rangle <  K / \epsilon \nu_{\min}$ is also a finite set. \cblue This means that the drift of the Lyapunov function $V(\cdot)$ is negative outside a finite set. \cblack. Then by using Foster-Lyapunov theorem, Markov chain $X(t)$ is positive recurrent and so the switch is stable.  

For the second requirement of SSC, \cblue we have to show that all the moments of $\norm{\q_{\perp \calk}}$ is bounded by constant that is independent of $\epsilon$, which we show by using part (i) of Lemma \ref{lemma: moment_bound}. \cblack From part (iii) of Lemma \ref{lemma: drift_results}, we have that for any $X$, the drift $\Delta W_{\perp \mathcal{K}} (X)$ satisfies condition C.2 with $D = n a_{\max}$. Now, we show that drift $ \Delta W_{\perp \mathcal{K}} (X)$ satisfies condition C.1.
We already know that 
\begin{equation*}
    \mathbb{E} \big[  \Delta V(X)\big| X(t)=X  \big] \leq K  - 2\epsilon\langle \q ,\boldsymbol  \nu \rangle - 2W_1 \norm{\mathbf{q}_{\perp \mathcal{K}}}.
\end{equation*}
Also, from Lemma \ref{lemma: drift_results} part (ii),
\begin{align*}
    \mathbb{E} \big[ \Delta V_{\| \mathcal{K}}(X) \big| X(t)=X \big] \allowdisplaybreaks \geq -2\epsilon \langle \q_{\| \mathcal{K}} ,\boldsymbol  \nu\rangle.
\end{align*}
By substituting the above equations in Lemma \ref{lemma: drift_results} part (iv), 
\begin{align*}
	 \mathbb{E} \big[ &\Delta W_{\perp \mathcal{K}}(X)\big| X(t)=X  \big] \\
	 &\leq \frac{1}{2\|\q_{\perp \mathcal{K}}\|} \Big ( K - 2\epsilon\langle \q_{\perp \mathcal{K}} ,\boldsymbol  \nu \rangle  - 2W_1 \norm{\q_{\perp \mathcal{K}}}  \Big )\\
	 &\cblue\leq \frac{1}{2\|\q_{\perp \mathcal{K}}\|} \Big ( K + 2\epsilon |\langle \q_{\perp \mathcal{K}} ,\boldsymbol  \nu \rangle|  - 2W_1 \norm{\q_{\perp \mathcal{K}}}  \Big ).\color{black}
\end{align*}
By Cauchy-Schwarz inequality, we get that $|\langle \q_{\perp \mathcal{K}} , \boldsymbol  \nu \rangle| \leq \norm{\q_{\perp \mathcal{K}}}\norm{\boldsymbol \nu}$. \cblue Note that the term $\langle \mathbf q_{\perp \mathcal{K}} , \boldsymbol  \nu \rangle$ might be positive because $ \mathbf q_{\perp \mathcal{K}}$ can have negative components. \color{black} By using this,
\begin{align*}
    \mathbb{E} \big[ \Delta W_{\perp \mathcal{K}}(X)\big| X(t)=X  \big] &\leq \frac{K}{2 \norm{\q_{\perp \mathcal{K}}}} +  \epsilon \norm{\boldsymbol \nu} - W_1 \\
    &\cblue \stackrel{(a)}{\leq} \frac{K}{2 \norm{\q_{\perp \mathcal{K}}}} - \frac{W_1}{2}\\
    & \stackrel{(b)}{\leq}  - \frac{W_1}{4},
\end{align*}
where $(a)$ follows when $\epsilon \leq \epsilon_0 = \frac{W_1}{2\norm{\boldsymbol \nu}}$ and (b) holds when $ \norm{\q_{\perp \mathcal{K}}} \geq 2 K/W_1$. This fulfils the condition C.2 with $\kappa = 2K/W_1$ and $\eta = W_1/4$. Then, by Lemma \ref{lemma: moment_bound} part (i), $\mathbb{E} \big[\norm{\Bar{\q}_{\perp \mathcal{K}}}^r \big] < C_r < \infty$, where $C_r$ can be derived by substituting the values of $\kappa$, $\eta$ and $D$ in Eq. \eqref{eq: part1_moment_bound} \cblue where $\kappa$, $\eta$ and $D$ are independent of $\epsilon$. \cblack So any scheduling algorithm in class $\Pi_1(\boldsymbol \nu)$ satisfies Eq. \eqref{eq: qperp_bound}. This completes the proof.
\end{proof}



\subsection{Appendix E}
\begin{proof}[Proof of Lemma \ref{lemma: expectation_bound} for power-of-$d$]
It is easy to observe that for power-of-$d$, the expected weight of the schedule increases as the value of $d$ increases. So, without loss of generality, we provide the proof only for power-of-$2$. \cblue One can do a more complicated analysis as compared to the one mentioned below to get a stronger bound, but for the results mentioned in this paper, we do not require such a strong bound. So we present a bound for power-of-$2$ and use the same lower bound for all values of $d$. \cblack

For the simplicity of notations, we drop the index $t$ in this proof. The weight the schedule chosen by  power-of-$2$ is,
\begin{align*}
    \langle \q,\s \rangle &= \max_{\s_1,\s_2} \{\langle \q,\s_1 \rangle, \langle \q,\s_2 \rangle  \}\\
    & = \langle \q,\s_1 \rangle + \max_{\s_1,\s_2} \{0, \langle \q,\s_2 \rangle -\langle \q,\s_1 \rangle  \},
\end{align*}
where $\s_1$ and $\s_2$ are sampled uniformly at random from $\mathcal X$. As $\s_1$ is a schedule chosen uniform at random, $\mathbb{E} [\langle \q,\s_1 \rangle] = \frac{1}{n} \langle \q,\mathbf 1 \rangle$.
Let $\Delta_{Po2}$ be the improvement in the weight (as compared to random scheduling) by doing power-of-$2$, then,
\begin{equation*}
	\Delta_{Po2}  = \max_{\s_1,\s_2} \{0, \langle \q,\s_2 \rangle -\langle \q,\s_1 \rangle  \}.
\end{equation*}
It follows that,
\begin{align} 
\label{eq: po2_eq1}
    \mathbb{E} \big[\Delta_{Po2} \big|\q \big] &= \frac{1}{(n!)^2} \sum_{s_1} \sum_{s_2} \max_{\s_1,\s_2} \{0, \langle \q,\s_2 \rangle -\langle \q,\s_1 \rangle  \} \nonumber \allowdisplaybreaks\\
    &= \frac{1}{(n!)^2} \sum_{s_1} \sum_{s_2}\frac{\langle \q,\s_2 \rangle - \langle \q,\s_1 \rangle}{2}\nonumber\\
    & \quad + \frac{1}{(n!)^2} \sum_{s_1} \sum_{s_2} \frac{|\langle \q,\s_1 \rangle - \langle \q,\s_2 \rangle|}{2}\nonumber \allowdisplaybreaks \\
    & = \frac{1}{2(n!)^2} \sum_{s_1} \sum_{s_2} \big|\langle \q,\s_1 \rangle - \langle \q,\s_2 \rangle \big|\nonumber \allowdisplaybreaks \\
    & \cblue \stackrel{(a)}{\geq} \frac{1}{2 (n!)} \sum_{\s_1} \Big |\frac{1}{n!} \sum_{\s_2} \langle \q,\s_1 \rangle - \frac{1}{n!} \sum_{\s_2} \langle \q,\s_2 \rangle \Big |\nonumber \allowdisplaybreaks \\
    & = \frac{1}{2 (n!)} \sum_{\s_1} \Big | \langle \q,\s_1 \rangle - \frac{1}{n!} \sum_{\s_2} \langle \q,\s_2 \rangle \Big |\nonumber \allowdisplaybreaks \\
    &\stackrel{(b)}{=} \frac{1}{2 (n!)} \sum_{\s_1} \Big | \langle \q,\s_1 \rangle - \frac{1}{n}\sum_{ij}q_{ij} \Big |\nonumber \allowdisplaybreaks \\
    &\stackrel{(c)}{=} \frac{1}{2 (n!)} \sum_{\s_1} \Big | \langle \q,\s_1 \rangle - \langle \q_{\|},\s_1 \rangle \Big |\nonumber \allowdisplaybreaks \\
    &= \frac{1}{2 (n!)} \sum_{\s_1} \Big | \langle \q_{\perp},\s_1 \rangle \Big | \allowdisplaybreaks,
\end{align}
where (a) follows by using triangle inequality, \cblue (b) follows because $\sum_{\s \in \mathcal{X}} \s = (n-1)! \mathbf{1}$, where $\mathcal{X}$ is the set of all possible schedules \color{black} and (c) follows by using Lemma \ref{lemma: projection_space} part (iii) as $\s_1 \in \mathcal F$. Now, consider a particular index $(i,j)$. Out of total $n!$ possible schedules, there would be $(n-1)!$ schedules with $s_{ij} = 1 $. Let $X(i,j)$ denote the set of all such schedules with $s_{ij} = 1 $. Note that we can partition $\mathcal{X}$ into $(n-1)!$ subsets, each containing $n$ schedule such that the entrywise sum of those schedule is $\mathbf{1}$. Similarly, we can partition $X(i,j)$ into $(n-2)!$ subsets (denoted by $X_k(i,j)$ for $k \in \{ 1, \dots, (n-2)!\}$), each of size $(n-1)$ such that 
\begin{equation*}
    \sum_{\s \in X_k(i,j)} \mathbf{s} = \mathbf{1} - \e^i -\Tilde{\e}^j + n\e^{ij},
\end{equation*}
where $\e^{ij}$ is a $n\times n$ matrix with 1 at index $(i,j)$ and all else 0. Then, for any $(i,j)$ \cblue we can just discard any schedule for which $s_{ij} =0$ to get \color{black}
\begin{align*}
     \sum_{\s} \Big | \langle \q_{\perp},\s \rangle \Big | &\geq \sum_{k=1}^{(n-2)!}\sum_{\s \in X_k(i,j)} \Big | \langle \q_{\perp},\s \rangle \Big | \allowdisplaybreaks\\
    &\geq \sum_{k=1}^{(n-2)!} \Big | \sum_{\s \in X_k(i,j)} \langle \q_{\perp},\s \rangle \ \Big | \allowdisplaybreaks\\
    &= \sum_{k=1}^{(n-2)!} \Big | \langle \q_{\perp},\mathbf{1}  -\e^i - \Tilde{\e}^j + n \e^{ij} \rangle \Big | \allowdisplaybreaks \\
    &\stackrel{(a)}{=}\sum_{k=1}^{(n-2)!} n | q_{\perp ij} | \allowdisplaybreaks\\
    &=  n \times  (n-2)! \times | q_{\perp ij} | \allowdisplaybreaks\\
    &\geq (n-1)!\times | q_{\perp ij} |
\end{align*}
where (a) follows from Lemma \ref{lemma: projection_space} part (ii). As this is true for all $(i,j)$ then we can take the element with highest value of $|q_{\perp ij}|$. Let $q_{\perp,\max} = \max_{i,j} |q_{\perp ij}|$, then
\begin{align}
\label{eq: po2_eq2}
    \frac{1}{n!}\sum_{\s} \Big | \langle \q_{\perp},\s \rangle \Big |& \geq  \frac{1}{n}\times q_{\perp,\max} \nonumber\\
     & \geq \frac{1}{n^2} \max_{\mathbf{s}} \langle \q_{\perp} , \mathbf{s} \rangle \allowdisplaybreaks \nonumber\\
    & = \frac{1}{n^2}  \max_{\mathbf{s}} \langle \q - \q_{\|} , \mathbf{s} \rangle \allowdisplaybreaks \nonumber\\
    & \stackrel{(a)} = \frac{1}{n^2}  \max_{\mathbf{s}} \langle \q , \mathbf{s} \rangle - \frac{1}{n^3} \langle \q,\mathbf{1} \rangle \allowdisplaybreaks \nonumber\\ 
    & \stackrel{(b)}{=} \frac{1}{n^2}  \max_{\mathbf{x} \in \mathcal{C}} \langle \q, \mathbf{x} \rangle - \frac{1}{n^3} \langle \q,\mathbf{1} \rangle  \allowdisplaybreaks \nonumber\\
    & \stackrel{(c)}{\geq} \frac{1}{n^2} \Big  \langle \q , \frac{1}{n} \mathbf{1} + \frac{\q_{\perp \mathcal{K}}}{n \norm{\q_{\perp \mathcal{K}}}} \Big \rangle - \frac{1}{n^3} \langle \q,\mathbf{1} \rangle  \allowdisplaybreaks \nonumber\\
    & = \frac{1}{n^2} \Big \langle \q , \frac{\q_{\perp \mathcal{K}}}{n \norm{\q_{\perp \mathcal{K}}}} \Big \rangle \allowdisplaybreaks \nonumber\\
    & =\cblue \frac{1}{n^2} \Big \langle \q_{\| \mathcal{K}} +\q_{\perp \mathcal{K}}, \frac{\q_{\perp \mathcal{K}}}{n \norm{\q_{\perp \mathcal{K}}}} \Big \rangle \allowdisplaybreaks \nonumber\\
    & \stackrel{(d)}{=}\cblue \frac{1}{n^2} \Big \langle \q_{\perp \mathcal{K}}, \frac{\q_{\perp \mathcal{K}}}{n \norm{\q_{\perp \mathcal{K}}}} \Big \rangle \allowdisplaybreaks \nonumber\\
    & = \frac{1}{n^3}  \norm{\q_{\perp \mathcal{K}}},
\end{align}
where (a) follows by using Lemma \ref{lemma: projection_space} part (iii), (b) follows because $\s$ lies in the set $\mathcal X$, which are also the extreme points of the polytope $\mathcal{C}$ by Birkhoff-von Neumann theorem, (c) follows from Lemma \ref{lemma: projection_cone} part (iv) by choosing $\mathbf{\nu} = \frac{1}{n} \mathbf{1}$, \cblue and (d) follows because $\langle \q_{\| \mathcal{K}},\q_{\perp \mathcal{K}} \rangle = 0$ by Lemma \ref{lemma: projection_cone} part (i) . Now, by substituting Eq. \eqref{eq: po2_eq2} in Eq. \eqref{eq: po2_eq1},
\begin{equation*}
     \mathbb{E} \big[\Delta_{Po2} \big|\q \big] \geq \frac{1}{2n^3}  \norm{\q_{\perp \mathcal{K}}}.
\end{equation*}
This completes the proof of Lemma \ref{lemma: expectation_bound} for power-of-$d$.
\end{proof}



\subsection{Appendix F}
\begin{proof}[Proof of Lemma \ref{lemma: expectation_bound} for random $d$-flip]

Similar to power-of-$d$, for random $d$-flip also, the expected weight of the schedule increases as the value of $d$ increases. So, without loss of generality, we provide the proof only for random $1$-flip. \cblue Similar to power-of-$d$, one can do a more complicated analysis to find a better bound for random $d$-flip, but for the results mentioned in this paper, we do not require such a bound. So, we only find the bound for $d=1$ and use it for all values of $d$. case For simplicity, we drop the index $t$ in the proof. \cblack

Suppose the indices chosen for the flip step are $(i,j)$ and $(k,l)$, then the weight of the schedule chosen by random $1$-flip scheduling is given by
	\begin{align*}
	    \langle \q,\s \rangle &= \max_{\s_1,\s_2} \{\langle \q,\s_1 \rangle, \langle \q,\s_2 \rangle  \}\\
	    & = \langle \q,\s_1 \rangle + \max_{\s_1,\s_2} \{0, \langle \q,\s_2 \rangle -\langle \q,\s_1 \rangle  \}\\ 
	    & =  \langle \q,\s_1 \rangle + \max\{0, q_{il}+ q_{kj} - q_{ij} - q_{kl}\},
	\end{align*}
	and $\mathbb{E} [\langle \q,\s_1 \rangle] = \frac{1}{n} \langle \q,\mathbf 1 \rangle$ as $\s_1$ is generated using random sampling. Let $\Delta_{flip}$ be the improvement in the weight (as compared to random scheduling) by doing random $1$-flip. Then 
	\begin{equation*}
		\Delta_{flip} = \max\{0, q_{il}+ q_{kj} - q_{ij} - q_{kl}\}.
	\end{equation*}
	Note that as the schedule $\s_1$ and the indices $(i,j)$ and $(k,l)$ for the flip step are chosen uniformly at random, any pair of indices has equal probability of getting selected. So, the probability of selecting $(i,j)$ and $(k,l)$ such that $ i\neq k$ and $l\neq j$ is $1/n^2(n-1)^2$. \cblue This is because there are $n$ possibilities for both $i$ and $j$, and there are $n-1$ possibilities for both $k$ and $l$ such that $k\neq i$ and $l\neq j$. \cblack Then,
	\begin{align}
	\label{eq: flip_eq1}
		\mathbb{E}[ & \Delta_{flip}|\q] \nonumber\\
		&= \frac{1}{n^2(n-1)^2} \sum_{ij} \sum_{ k\neq i} \sum_{l \neq j} \max\{0, q_{il}+ q_{kj} - q_{ij} - q_{kl}\}\allowdisplaybreaks \nonumber\\
		& \geq \frac{1}{n^4}\sum_{ij}  \sum_{ k\neq i} \sum_{l \neq j}  \max\{0, q_{il}+ q_{kj} - q_{ij} - q_{kl}\}\allowdisplaybreaks \nonumber\\
		&=\frac{1}{2n^4} \sum_{ij}  \sum_{ k\neq i} \sum_{l \neq j}  q_{il}+ q_{kj} - q_{ij} - q_{kl}\allowdisplaybreaks \nonumber\\
		& \ + \frac{1}{2n^4} \sum_{ij}  \sum_{ k\neq i} \sum_{l \neq j}  |q_{il}+ q_{kj} - q_{ij} - q_{kl}|\allowdisplaybreaks \nonumber\\
		&\stackrel{(a)}{=} \frac{1}{2n^4} \sum_{ij}  \sum_{ k\neq i} \sum_{l \neq j}  | q_{ij} + q_{kl}-q_{il}- q_{kj}| \allowdisplaybreaks  \nonumber \\
		&\geq \frac{1}{2n^4} \sum_{ij} \big| \sum_{ k\neq i} \sum_{l \neq j}   q_{ij} + q_{kl}-q_{il}- q_{kj} \big|,
	\end{align}
	 \cblue where (a) holds by using following calculation, \color{black}
	\begin{align*}
		 \sum_{ k\neq i} \sum_{l \neq j}  q_{ij} & =  (n-1)^2 q_{ij}.\allowdisplaybreaks\\
		 \sum_{ k\neq i} \sum_{l \neq j} q_{kl} & =  \sum_{kl} q_{kl} - \sum_{k} q_{kj} - \sum_{l} q_{il} + q_{ij}.\allowdisplaybreaks\\
		 \sum_{ k\neq i} \sum_{l \neq j}  q_{il} & =  (n-1)(\sum_{l} q_{il} - q_{ij}).\allowdisplaybreaks\\
		 \sum_{ k\neq i} \sum_{l \neq j}  q_{kj} & =  (n-1)(\sum_{k} q_{kj} - q_{ij}).
	\end{align*} 
	By combining the above equations, we get that,
	\begin{align*}
		 \sum_{ k\neq i} & \sum_{l \neq j}  q_{ij} + q_{kl}-q_{il}- q_{kj} \\
		&= n^2(q_{ij} - \frac{1}{n} \sum_{l} q_{il} - \frac{1}{n} \sum_{k} q_{kj} + \frac{1}{n^2} \sum_{kl} q_{kl} ),
	\end{align*}
	\cblue
	then by adding over $(i,j)$,
	\begin{equation*}
	  \sum_{ij}  \sum_{ k\neq i} \sum_{l \neq j}  q_{ij} + q_{kl}-q_{il}- q_{kj} =0.
	\end{equation*}
	Also, by using Lemma \ref{lemma: projection_space} part (i),
	\begin{align*}
		 \sum_{ k\neq i} & \sum_{l \neq j}  q_{ij} + q_{kl}-q_{il}- q_{kj} \\
		&= n^2(q_{ij} - \frac{1}{n} \sum_{l} q_{il} - \frac{1}{n} \sum_{k} q_{kj} + \frac{1}{n^2} \sum_{kl} q_{kl} )  \allowdisplaybreaks\\
		& = n^2 q_{\perp ij},
	\end{align*}
	\color{black}
	Substituting this in Eq. \eqref{eq: flip_eq1} gives us,
	\begin{align*}
		\mathbb{E}[\Delta_{flip}|\q] &\geq \frac{1}{2n^2} \sum_{ij} |q_{\perp ij} | \allowdisplaybreaks \\
		& \geq \frac{1}{2n^2} \max_{\s} \langle \q_{\perp} , \s \rangle   \stackrel{(a)} \geq  \frac{1}{2n^3} \norm{\q_{\perp \mathcal{K}}},
	\end{align*}
	where (a) follows in similar way as Eq. \eqref{eq: po2_eq2}. This completes the proof of Lemma \ref{lemma: expectation_bound} for random $d$-flip.
	\end{proof}


\subsection{Appendix G}
\label{appendix: class_2}

\begin{proof}[Proof of Theorem \ref{thm: ssc_class_2}]
For a scheduling algorithm in class $\Pi_2$, consider the underlying Markov chain $X(t)$. By definition of class $\Pi_2$,  the Markov chain $X(t)$ follows condition A.1 and A.2. Recall that by condition A.2, there exists $g(\cdot)$ such that $\q(t) = g(X(t))$ for all $t\geq 0$. Take $\q = g(X)$.
\cblue By Definition \ref{Def:class_2}, 
\begin{equation*}
    	\mathbb{E} \big[ \langle \q(t),\s(t) \rangle \big| X(t) = X \big] \geq \max_{\s} \langle \q,\s \rangle - W_2.
\end{equation*}
By substituting this in part (i) of Lemma \ref{lemma: drift_results}, for the Lyapunov function $V(X)$ defined in Lemma \ref{lemma: drift_results}, \color{black}
\begin{align}
\label{eq: apx_thm_eq}
    \mathbb{E} \big[ & \Delta V(X) \big| X(t)=X  \big] \nonumber\\
     & \leq K  + 2(1-\epsilon)\langle \q , \boldsymbol \nu \rangle - 2 \mathbb{E} \big[ \langle \q(t),\s(t) \rangle \big| X(t) = X \big] \nonumber \allowdisplaybreaks \\
   & \stackrel{(a)} \leq  K  + 2(1-\epsilon)\langle \q ,\boldsymbol  \nu \rangle - 2\max_{s} \langle \q , \s \rangle + 2W_2 \nonumber \allowdisplaybreaks\\
   &\cblue \stackrel{(b)} \leq  K'  + 2(1-\epsilon)\langle \q ,\boldsymbol  \nu \rangle - 2\max_{s} \langle \q , \s \rangle \nonumber \allowdisplaybreaks\\
   &\stackrel{(c)} \leq K' - 2\epsilon \langle \q ,\boldsymbol  \nu \rangle - 2\nu_{\min}\norm{\q_{\perp \mathcal{K}}},
\end{align}
 where (a) follows from the definition of class $\Pi_2$, (b) follows by taking  $K' =K+2W_2$ and (c) follows by using Lemma \ref{lemma: projection_cone} part (iv). In order to prove the stability of algorithms in class $\Pi_2$, we use $\langle \q ,\boldsymbol  \nu \rangle \geq \nu_{\min} \langle \q , \mathbf 1 \rangle$ to get 
 \begin{align*}
 	\mathbb{E} \big[  \Delta V(X)\big| X(t)=X \big] & \leq K' - 2\epsilon \nu_{\min} \langle \q , \mathbf 1 \rangle\\
 	& \stackrel{(a)} \leq - \epsilon \nu_{\min} \langle \q , \mathbf 1 \rangle, 
 \end{align*}
 where $(a)$ holds whenever $\langle \q , \mathbf 1 \rangle \geq  K' / \epsilon \nu_{\min}$.  
 It is easy to observe that the set of $\q$ for which $\langle \q , \mathbf 1 \rangle <  K' / \epsilon \nu_{\min}$ forms a finite set whenever $\nu_{\min}>0$. Then, by condition A.2, the set of $X$ for which $\langle \q , \mathbf 1 \rangle <  K' / \epsilon \nu_{\min}$ is also finite.  Thus, by Foster-Lyapunov theorem, we get that the Markov chain $X(t)$ is positive recurrent and so the system is stable. 
 
 Now we prove that the scheduling algorithm in class $\Pi_2$ satisfies SSC. \cblue We have to show that all the moments of $\norm{\q_{\perp \calk}}$ is bounded by constant that is independent of $\epsilon$, which we show by using part (i) of Lemma \ref{lemma: moment_bound}. \cblack Note that the condition C.2 for the Lyapunov function $W_{\perp \mathcal{K}} (\cdot)$ is directly satisfied by using Lemma \ref{lemma: drift_results} part (iii). So we just need to prove that $W_{\perp \mathcal{K}} (\cdot)$ satisfies condition C.1. From Lemma \ref{lemma: drift_results} part (ii),
 \begin{align*}
    \mathbb{E} \big[ & \Delta V_{\| \mathcal{K}}(X)\big| X(t)=X \big] \allowdisplaybreaks \geq -2\epsilon \langle \q_{\| \mathcal{K}} ,\boldsymbol  \nu\rangle.
\end{align*}
By substituting this and Eq. \eqref{eq: apx_thm_eq} in Lemma \ref{lemma: drift_results} part (iv),
\begin{align*}
    \mathbb{E} \big[ & \Delta W_{\perp \mathcal{K}} (X)\big| X(t)=X \big] \\
    &\leq \frac{1}{2\|\q_{\perp \mathcal{K}}\|} \Big ( K' - 2\epsilon\langle \q_{\perp \mathcal{K}} , \boldsymbol \nu \rangle  - 2\nu_{\min} \norm{\q_{\perp \mathcal{K}}}  \Big )\\
    &\cblue\leq \frac{1}{2\|\q_{\perp \mathcal{K}}\|} \Big ( K' + 2\epsilon |\langle \q_{\perp \mathcal{K}} ,\boldsymbol  \nu \rangle|  - 2\nu_{\min} \norm{\q_{\perp \mathcal{K}}}  \Big )\color{black}\\
    &\stackrel{(a)}\leq \frac{K'}{2\|\q_{\perp \mathcal{K}}\|} + \epsilon \norm{\boldsymbol \nu} - \nu_{\min}\\
    &\stackrel{(b)} \leq - \frac{\nu_{\min}}{4}, 
\end{align*}
where (a) follows by the using Cauchy-Schwarz inequality, (b) follows whenever $\epsilon \leq \epsilon_0 = \nu_{\min}/2\norm{\boldsymbol \nu}$ and $\|\q_{\perp \mathcal{K}}\| \geq 2K'/\nu_{\min}$. This fulfils the condition C.1 with $\kappa  = 2K'/\nu_{\min}$ and $\eta = \nu_{\min}/4$. Then, by Lemma \ref{lemma: moment_bound} part (i), $\mathbb{E} \big[\norm{\Bar{\q}_{\perp \mathcal{K}}}^r \big] < C_r < \infty$, where $C_r$ can be derived by substituting the values of $\kappa$, $\eta$ and $D$ in Eq. \eqref{eq: part1_moment_bound}, \cblue where $\kappa$, $\eta$ and $D$ are independent of $\epsilon$. \cblack This proves that any scheduling algorithm in class $\Pi_2$ satisfies Eq. \eqref{eq: qperp_bound}.
\end{proof}

\subsection{Appendix H}
\label{appendix: markov_chain}

\begin{proof}[Proof of condition A.1 and A.2 for bursty MaxWeight]
    Let $m(t)$ be the counter that denotes the number of time slots since the system calculated the MaxWeight schedule. Then, the corresponding Markov chain is $X(t) =\big(\q(t), \s(t),m(t) \big) $. One issue with bursty MaxWeight is that $X(t)$ is periodic with period $m$. \cblue This happens because the counter $m(t)$ resets after exactly $m$ time slots. Thus the counter $m(t)$ is periodic with period $m$, which in turn makes $X(t)$ periodic with period $m$. \color{black} To make it aperiodic, we consider a slight modification. We assume that if $\q(t) = \mathbf{0}$, then the counter $m(t)$ also becomes $0$. 

	 Then, the Markov chain $X(t)$ is aperiodic because the state $X = (\mathbf 0, \mathbf I, 0)$ has a self loop \cblue by using Assumption \ref{assu: arrival_process} part (iv) and the fact that $m(t) = 0$ if $\q(t) =\mathbf 0$. \cblack The irreducibility of $X(t)$ follows by considering the communicating class of $X = (\mathbf 0, \mathbf I, 0)$ and using similar arguments as in \cite[Exercise 4.2]{srikant2014communication}. This means $X(t)$ satisfies condition A.1. Now, condition A.2 holds simply because $\s(t)$ and $ m(t)$ can take only a finite number of values. Thus the Markov chain $X(t)$ satisfy condition A.1 and A.2. 
\end{proof}

\begin{proof}[Proof of condition A.1 and A.2 for Pipelined MaxWeight]
    The corresponding Markov chain for pipelined MaxWeight is $X(t) = \big(\q(t), \q(t-m) \big)$, where $ \q(t-m) = \q(0)$ for any $t <m$. The Markov chain $X(t)$ is aperiodic because the state $X = (\mathbf 0, \mathbf 0)$ has a self loop by using Assumption \ref{assu: arrival_process} part (iv). Irreducibility follows by using similar arguments as in \cite[Exercise 4.2]{srikant2014communication}. This gives us that $X(t)$ satisfies condition A.1. Now, note that for any $(i,j)$, and for any $t$,
 \begin{equation*}
 	q_{ij}(t) - n a_{\max} \leq q_{ij}(t-m) \leq q_{ij}(t) +n.
 \end{equation*}
\cblue From the above expression, we observe that if $\q(t)$ lies in a finite set, then $\q(t-m)$ also lies in a finite set, and the condition A.2 is satisfied. \cblack
 Thus the Markov chain $X(t)$ satisfy condition A.1 and A.2. 
\end{proof}

\subsection{Appendix I}
\label{appendix: lemma_tau_dominate}

\begin{proof}[Proof of Lemma \ref{prop:tau_dominate}]
For any $Y$, we have that
\begin{align*}
	\pr (& \tau_k > c | Y_k = Y) \\
	&= \pr (\tau_k > c-1 | Y_k = Y)\pr (\tau_k > c | \tau_k > c-1 , Y_k = Y).
\end{align*}
Let $A_t$ be the event that 
\begin{equation*}
	\langle \q(t), \s(t) \rangle \neq \max_{\s} \langle \q(t) , \s \rangle.
\end{equation*}
Then, by Definition \ref{Def: class_3} part (i), \cblue we can condition on the possible values of $T_k$ to get \cblack
\begin{align*}
	& \pr( \tau_k > c | \tau_k > c-1 , Y_k = Y) \nonumber \\
	& = \sum_{T\geq 0} \pr(T_k = T) \pr(A_{T+c} |\tau_k > c-1 , Y_k = Y, T_k =T),
\end{align*}
\cblue where the above equality holds because for $\tau_k > c$, we need that $t_k > c-1$ and the MaxWeight schedule is not used at time $T_k +c$. \cblack Now,
\begin{align}
	\label{eq: tau_1}
	\pr( & A_{T+c} |\tau_k > c-1 , Y_k = Y, T_k =T) \nonumber \\
	&= \ex \big[ \cali (A_{T+c}) \big|\tau_k > c-1 , Y_k = Y, T_k =T \big]\nonumber\\
	&\stackrel{(a)} = \ex \Big[ \ex \big[ \cali (A_{T+c}) \big| \mathcal H_{T+c} \big] \Big|\tau_k > c-1 , Y_k = Y, T_k =T \Big]\nonumber\\
	& \stackrel{(b)}\leq \ex \Big[ (1-\delta) \Big|\tau_k > c-1 , Y_k = Y, T_k =T \Big]\nonumber \\
	& = (1-\delta).
\end{align}
where (a) follows because the event $\{\tau_k > c-1 , Y_k = Y, T_k =T\}$ is a function of $\mathcal H_{T+c}$ \cblue (Recall that $\mathcal{H}_t$ is the history till $t$ as given in Eq. \eqref{eq: filtration_history}) \cblack and (b) follows by using part (i) of Definition \ref{Def: class_3}. 
Thus,
\begin{align*}
	\pr ( & \tau_k > c | Y_k = Y)\\ 
	&\leq \pr (\tau_k > c-1 | Y_k = Y) \times (1-\delta) \sum_{T\geq 0} \pr(T_k = T) \\
	& = \pr (\tau_k > c-1 | Y_k = Y) \times (1-\delta).
\end{align*}
Now, by using the above relation iteratively, we get that,
\begin{align*}
	\pr(\tau_k > c | Y_k=Y) \leq (1-\delta)^c = \pr(M >c),
\end{align*}
where $M$ is a geometric random variable with mean $1/\delta$.
\end{proof}

\subsection{Appendix J}
\begin{proof}
As mentioned before the stability of algorithms in class $\Pi_3$ follows from the arguments in \cite{665071}. So, we just prove that algorithms in class $\Pi_3$ achieves SSC. We redefine the Lyapunov functions as follows
	 \begin{align*}
	 V(Y) = \norm{\q}^2 && V_{\| \mathcal{K}}(Y)  =  \norm{ \q_{\| \mathcal{K}}}^2 && W_{\perp \calk}(Y) = \norm{\q_{\perp \calk}}.
	 \end{align*}
First, we prove that the drift of Lyapunov function $\Delta W_{\perp \calk}(Y)$ satisfy the condition C.3,
	\begin{align*}
		\big| & \Delta  W_{\perp \calk}(Y) \big| \allowdisplaybreaks \\
		 &= \big| \norm{\q_{\perp \calk}(T_{k+1})} - \norm{\q_{\perp \calk}(T_{k})} \big| \cali (Y_k = Y) \allowdisplaybreaks \\
		 & \leq \big| \norm{\q_{\perp \calk}(T_{k+1})} - \norm{\q_{\perp \calk}(T_{k})} \big| \allowdisplaybreaks \\
		& = \Big| \sum_{i=T_k}^{T_{k+1}-1} \norm{\q_{\perp \calk}(i+1)} - \norm{\q_{\perp \calk}(i)} \Big| \allowdisplaybreaks\\
		&\leq \sum_{i=T_k}^{T_{t+1}-1}  \Big| \norm{\q_{\perp \calk}(i+1)} - \norm{\q_{\perp \calk}(i)} \Big| \allowdisplaybreaks \\
		& \stackrel{(a)}\leq \sum_{i=T_k}^{T_{t+1}-1} n a_{\max}\\
		& = na_{\max} \tau_k,
	\end{align*}
	where (a) follows from Lemma \ref{lemma: drift_results} part (iii). Let $M$ be a geometric random variable with mean $1/\delta$. From Lemma \ref{prop:tau_dominate}, we know that $\tau_k$ is stochastically dominated by the random variable $M$. This implies that $|\Delta W_{\perp \calk}(Y) |$ is stochastically dominated by random variable $na_{\max}M$. 	Also, by the  property of geometric random variable, the Moment Generating Function (MGF) of $na_{\max}M$ exists and is finite. This proves that Lyapunov function $\Delta W_{\perp \calk}(Y)$ satisfy the condition C.3. Next we look at the drift equations. The process $\q(T_k)$ evolves according to following equation,
	 \begin{equation*}
	 	\q(T_{k+1}) = \q(T_{k}) + \sum_{i=0}^{\tau_k -1} \Big( \ar(T_k+i) -\s(T_k+i) + \mathbf{u}(T_k+i)\Big)
	 \end{equation*}
	 	 This gives us that,
	 \begin{align}
	 \label{eq: q_expand}
	 	\| & \q(T_{k+1}) \|^2 \nonumber\\
	 	 = & \norm{ \q(T_{k}) + \sum_{i=0}^{\tau_k -1} \Big( \ar(T_k+i) -\s(T_k+i)\Big)}^2  \allowdisplaybreaks \nonumber \\
	 	 & + \norm{\sum_{i=0}^{\tau_k -1} \mathbf{u}(T_k+i)}^2 \allowdisplaybreaks \nonumber\\
	 	& + 2 \Big \langle \sum_{i=0}^{\tau_k -1} \mathbf{u}(T_k+i), \q(T_{k+1}) - \sum_{i=0}^{\tau_k -1} \mathbf{u}(T_k+i)\Big \rangle \allowdisplaybreaks\nonumber\\
	 	\leq & \norm{ \q(T_{k}) + \sum_{i=0}^{\tau_k -1} \Big( \ar(T_k+i) -\s(T_k+i)\Big)}^2 \allowdisplaybreaks\nonumber\textbf{}\\
	 	&+ 2 \sum_{i=0}^{\tau_k -1} \Big \langle  \mathbf{u}(T_k+i), \q(T_{k+1}) \Big \rangle \allowdisplaybreaks\nonumber\\
	 	\leq & \norm{ \q(T_{k})}^2 + \norm{ \sum_{i=0}^{\tau_k -1} \ar(T_k+i)}^2 + \norm{ \sum_{i=0}^{\tau_k -1} \s(T_k+i)}^2\allowdisplaybreaks\nonumber\\
	 	&+ 2 \Big \langle \q(T_k) , \sum_{i=0}^{\tau_k -1} \Big( \ar(T_k+i) -\s(T_k+i)\Big) \Big \rangle \allowdisplaybreaks\nonumber \\
	 	&+ 2 \sum_{i=0}^{\tau_k -1} \Big \langle  \mathbf{u}(T_k+i), \q(T_{k+1}) \Big \rangle
	 \end{align}
	 
	 Now we provide the a bound for the expected value of each of the term above. Then,
\begin{align*}
	 	\ex \big[ \langle \q & (T_k+i+2),\mathbf{u}(T_k+i)\rangle \big| Y_k =Y, \tau_k \big]\\
	 	\leq & \cblue \ex \big[ \langle \q(T_k+i+1),\mathbf{u}(T_k+i)\rangle \big| Y_k=Y, \tau_k \big]\\
	 	&\cblue+ \ex\big[ \langle \mathbf{a}(T_k+i+1),\mathbf{u}(T_k+i)\rangle \big| Y_k =Y, \tau_k \big]\\
	 	&\cblue- \ex\big[ \langle \mathbf{s}(T_k+i+1),\mathbf{u}(T_k+i)\rangle \big| Y_k =Y, \tau_k \big]\\
	 	&\cblue+ \ex\big[ \langle \mathbf{u}(T_k+i+1),\mathbf{u}(T_k+i)\rangle \big| Y_k =Y, \tau_k \big]\\
	 \stackrel{(a)}{\leq} & \ex \big[ \langle \ar(T_k+i+1),\mathbf{u}(T_k+i)\rangle \big| Y_k=Y, \tau_k \big]\\
	 	&+ \ex\big[ \langle \mathbf{u}(T_k+i+1),\mathbf{u}(T_k+i)\rangle \big| Y_k =Y, \tau_k \big]\\
	 	\leq & \cblue \sum_{ij}a_{\max}u_{ij}(T_k+i+1)  + \sum_{ij}u_{ij}(T_k+i+1) \\
	 	\leq  & na_{\max} + n\\
		\leq & 2na_{\max},
	 \end{align*}
	 \cblue where (a) holds because $\langle \q(t+1),\mathbf{u}(t)\rangle = 0 $ and $\langle \s(t+1),\mathbf{u}(t)\rangle = 0 $. \cblack Using this iteratively, we get that
	 \begin{equation*}
	 	\ex[  \langle \q(T_{k+1}),\mathbf{u}(T_k+i)\rangle | Y_k = Y, \tau_k ] \leq 2na_{\max} \times (\tau_k - i -1).
	 \end{equation*}
	 	 Thus,
	 \begin{align}
	 \label{eq: u_q_inner}
	 	\ex \Big[ & \sum_{i=0}^{\tau_k -1} \Big \langle  \mathbf{u}(T_k+i), \q(T_{k+1}) \Big \rangle \Big| Y_k=Y \Big] \nonumber \\
	 	&= \ex \Big[ \sum_{i=0}^{\tau_k -1} \ex \big[ \langle  \mathbf{u}(T_k+i), \q(T_{k+1}) \rangle \big| Y_k=Y, \tau_k \big] \Big| Y_k =Y\Big] \nonumber\\
	 	& \leq  \ex \Big[ \sum_{i=0}^{\tau_k -1} 2na_{\max} \times (\tau_k - i -1) \Big| Y_k=Y \Big]\nonumber \\
	 	&\leq  \frac{2na_{\max}}{2}\ex [\tau_k^2 | Y_k=Y]\nonumber\\
	 	&\leq  \frac{2na_{\max}}{\delta^2},
	 \end{align}
	 \cblue where the last inequality holds because of Lemma \ref{prop:tau_dominate}. \cblack
	 Next, we use the use the fact that arrivals are bounded to get the following result,
	 	\begin{align}
	 	\label{eq: norm_ar}
	 	\ex \Big[ & \norm{ \sum_{i=0}^{\tau_k -1} \ar(T_k+i)}^2 \Big| Y_k=Y \Big] \nonumber\\
	 	& = \ex \Big[ \sum_{i=0}^{\tau_k -1} \sum_{j=0}^{\tau_k -1} \big \langle \ar(T_k+i) , \ar(T_k+j) \big \rangle \Big | Y_k =Y\Big]\nonumber\\
	 	& \stackrel{(a)} \leq n^2 a_{\max}^2 \ex [\tau_k^2 |Y_k=Y]\nonumber\\
	 	& \leq \frac{2n^2 a_{\max}^2}{\delta^2},
\end{align}
where (a) follows from the fact that for any $t,t'$, $ \langle \ar(t),\ar(t') \rangle \leq n^2 a_{\max}^2 $.

By using the fact that the schedules are permutation matrices, 
	 \begin{align}
	  \label{eq: norm_sur}
	 	\mathbb{E} \Big[ & \norm{ \sum_{i=0}^{\tau_k -1} \s(T_k+i)}^2 \Big| Y_k=Y \Big] \nonumber\\
	 	& = \ex \Big[ \sum_{i=0}^{\tau_k -1} \sum_{j=0}^{\tau_k -1} \big \langle \s(T_k+i) , \s(T_k+j) \big \rangle \Big | Y_k=Y \Big]\nonumber\\
	 	& \stackrel{(a)} \leq n \ex [\tau_k^2 |Y_k=Y]\nonumber\\
	 	& \leq \frac{2n}{\delta^2},
	\end{align}
	where (a) follows from the fact that for any $t,t'$, $ \langle \s(t),\s(t') \rangle \leq n $.
Now, by $Y_k = Y$ and $(\q,\s) = f(Y)$, we get that 
	 \begin{align}
	 \label{eq: a_q_dot}
	 	\ex \Big[ & \sum_{i=0}^{\tau_k -1} \Big \langle  \ar(T_k+i), \q(T_{k}) \Big \rangle \Big| Y_k =Y\Big] \allowdisplaybreaks \nonumber \\
	 	&=\ex \Big[ \sum_{i=0}^{\tau_k -1} \Big \langle  \ar(T_k+i), \q \Big \rangle \Big| Y_k =Y \Big] \allowdisplaybreaks \nonumber\\
	 	& \stackrel{(a)} = \langle \q, \boldsymbol \lambda \rangle \times \ex[\tau_k | Y_k =Y ],
	 \end{align}
	 where (a) follows because for any $m$, the event $\{\tau_k = m\}$ does not depend on $\{\ar(T_k +i)\}_{i\geq m}$, and so we can use the general form of Wald's equation \cite{blackwell1946equation}.

By Assumption \ref{assu: arrival_process}, we know that for any $(i,j)$, $a_{ij} (t) \leq a_{\max}$ and so $q_{ij}(t+1) \leq q_{ij}(t) +a_{\max}$. Thus, for any $k$, $q_{ij}(t+k) \leq q_{ij}(t) +ka_{\max}$. This gives us that
\begin{align}
\label{eq: next_eq}
	\ex \big[ & \langle \q(T_k),  \s(T_k +i ) \rangle \big|Y_k  =Y  \big] \nonumber\\
	  & \geq \ex \big[ \langle \q(T_k+i), \s(T_k +i ) \rangle \big| Y_k  =Y  \big]   - i\times na_{\max}.
\end{align}	
	 Also, as the scheduling algorithm lies in class $\Pi_3$, 
	 \begin{align}
	 \label{eq: class3_1}
	 	\ex \big[ \langle \q & (T_k+i), \s(T_k +i ) \rangle \big|Y_k  =Y  \big] \allowdisplaybreaks \nonumber\\
	 	\stackrel{(a)}\geq & \ex \big[ \langle \q(T_k+i), \s(T_k +i - 1)\rangle  \big|Y_k  =Y  \big] \allowdisplaybreaks \nonumber\\
	 	 \stackrel{(b)} \geq & \ex \big[ \langle \q(T_k+i-1), \s(T_k +i - 1)\rangle  \big|Y_k  =Y  \big] - n  \allowdisplaybreaks \nonumber\\
	 	 \stackrel{(c)} \geq &\ex \big[ \langle \q(T_k), \s(T_k)\rangle  \big|Y_k  =Y  \big] - i\times n.
	 \end{align}
where (a) follows by Eq. \eqref{eq: comparison} and (b) follows because for all $t \geq 0$, $q_{ij}(t+1) \geq q_{ij}(t) - s_{ij}(t) \geq q_{ij}(t) - 1$. Finally, (c) follows by using the inequality in (b) repeatatively.
By combining Eq. \eqref{eq: class3_1} with Eq. \eqref{eq: next_eq}, we get,
	 \begin{align*}
	 	\ex  \big[\langle & \q(T_k), \s(T_k +i ) \rangle \big|Y_k  =Y  \big]\\
	 	 &\geq \ex \big[ \langle \q(T_k), \s(T_k) \rangle \big|Y_k   =Y \big] - i \times n (1+a_{\max}).
	 \end{align*}
This gives us that, 
	 \begin{align}
	 \label{eq: s_q_dot}
	 	\ex \Big[ & \sum_{i=0}^{\tau_k -1} \Big \langle   \q(T_{k}),\s(T_k+i) \Big \rangle \Big| Y_k = Y\Big] \allowdisplaybreaks \nonumber \\
	 	&\geq\ex \Big[ \sum_{i=0}^{\tau_k -1} \Big( \langle   \q(T_k), \s(T_k) \rangle - i\times n(1+a_{\max}) \Big)  \Big| Y_k= Y \Big] \allowdisplaybreaks\nonumber \\
	 	&\geq  \langle \q, \s \rangle \times  \ex[\tau_k | Y_k =Y] - \frac{n(1+a_{\max})}{2} \times \ex [\tau_k^2 | Y_k =Y]  \allowdisplaybreaks\nonumber\\
	 	&  \geq \langle \q, \s \rangle \times \ex[\tau_k | Y_k =Y] - \frac{n(1+a_{\max})}{\delta^2},
	 \end{align}
By putting Eq. (\ref{eq: u_q_inner}), (\ref{eq: norm_ar}), (\ref{eq: norm_sur}), (\ref{eq: a_q_dot}) and (\ref{eq: s_q_dot}) in Eq. (\ref{eq: q_expand}), we get that the drift of $V(Y_k)$, when $Y_k = Y$, is given by,
	  \begin{align*}
	 	\ex \big[ \Delta & V(Y) \big| Y_k =Y \big] \\
	   \stackrel{(a)}	\leq & K+ 2 \left( \langle \q,\boldsymbol  \lambda \rangle - \langle \q, \s \rangle \right) \times \ex[\tau_k | Y_k =Y]\\
	   \stackrel{(b)}	\leq & K -2 \left( \epsilon \langle \q,\boldsymbol  \nu \rangle + \nu_{\min}\norm{\q_{\perp \calk}}  \right) \times \ex[\tau_k | Y_k =Y]\\
	    \stackrel{(c)}	\leq & K - 2\epsilon \langle \q,\boldsymbol  \nu \rangle  \times \ex[\tau_k | Y_k =Y] - 2\nu_{\min}\norm{\q_{\perp \calk}}  
	 \end{align*}
	 where (a) follows by taking
	 \begin{equation*}
	 	K = \frac{2n^2a_{\max}^2+ 6na_{\max} +4n}{\delta^2},
	  \end{equation*}
	  (b) follows by taking $\boldsymbol \lambda = (1-\epsilon) \nu$ and by construction of Markov chain $Y_k$, $\s$ is the MaxWeight schedule with respect to $\q$, so we can use Lemma \ref{lemma: projection_cone} part (iv). Finally, (c) follows because $\tau_k \geq 1$.
	  
	 Now for the drift of $V_{\| \calk}(Y_k)$, 
	 	 \begin{align*}
	 	\ex & \big[ \Delta V_{\| \calk}(Y) \big| Y(t) =Y \big] \allowdisplaybreaks \\
	 	& = \ex \big[ \norm{\q_{\| \calk}(T_{k+1})}^2 - \norm{\q_{\| \calk}(T_k)}^2 \big | Y_k =Y  \big] \allowdisplaybreaks \\
	 	&= \ex \big[ \norm{\q_{\| \calk}(T_{k+1}) - \q_{\| \calk}(T_k)}^2 \big | Y_k =Y \big] \allowdisplaybreaks \\
	 	& \qquad + 2 \ex \big[  \langle \q_{\| \calk}(T_k), \q_{\| \calk}(T_{k+1}) - \q_{\| \calk}(T_k) \rangle \big | Y_k =Y \big] \allowdisplaybreaks \\
	 	& \geq  2 \ex \big[  \langle \q_{\| \calk}(T_k), \q_{\| \calk}(T_{k+1}) - \q_{\| \calk}(T_k) \rangle \big | Y_k =Y \big] \allowdisplaybreaks \\
	 	& =  2 \ex \big[  \langle \q_{\| \calk}(T_k), \q(T_{k+1}) - \q(T_k) \rangle \big | Y_k =Y \big] \allowdisplaybreaks \\
	 	& \qquad - 2 \ex \big[  \langle \q_{\| \calk}(T_k), \q_{\perp \calk}(T_{k+1}) - \q_{\perp \calk}(T_k) \rangle \big | Y_k =Y \big] \allowdisplaybreaks \\
	 	&\stackrel{(a)} \geq 2 \ex \big[  \langle \q_{\| \calk}(T_k), \q(T_{k+1}) - \q(T_k) \rangle \big | Y_k =Y \big] \allowdisplaybreaks \\
	 	& = 2\ex \Big [ \Big \langle \q_{\| \calk}(T_k), \sum_{i = 0}^{\tau_k -1} \Big ( \ar(T_k+i) - \s(T_k+i) \Big ) \Big \rangle \Big | Y_k =Y \Big] \allowdisplaybreaks \\
	 	& \qquad + 2\ex \Big [ \Big \langle \q_{\| \calk}(T_k), \sum_{i = 0}^{\tau_k -1} \mathbf u(T_k +i) \Big \rangle \Big | Y_k =Y \Big] \allowdisplaybreaks \\
	 	& \stackrel{(b)} \geq 2\ex \Big [ \Big \langle \q_{\| \calk}(T_k), \sum_{i = 0}^{\tau_k -1} \Big ( \ar(T_k+i) - \s(T_k+i) \Big ) \Big \rangle \Big | Y_k =Y \Big] \allowdisplaybreaks \\
	 	& = 2\ex \Big [ \Big \langle \q_{\| \calk}(T_k), \sum_{i = 0}^{\tau_k -1} \big ( \ar(T_k+i) - \boldsymbol \nu \big ) \Big \rangle \Big | Y_k =Y \Big] \allowdisplaybreaks \\
	 	& \qquad -   2\ex \Big [ \Big \langle \q_{\| \calk}(T_k), \sum_{i = 0}^{\tau_k -1} \s(T_k+i) - \boldsymbol \nu \Big \rangle \Big | Y_k =Y \Big] \allowdisplaybreaks \\
	 	& \stackrel{(c)} = 2\ex \Big [ \Big \langle \q_{\| \calk}(T_k), \sum_{i = 0}^{\tau_k -1} \big ( \ar(T_k+i) -\boldsymbol  \nu \big ) \Big \rangle \Big | Y_k =Y \Big] \allowdisplaybreaks \\
	 	& \stackrel{(d)}=  -2\epsilon \langle \q_{\| \calk},\boldsymbol  \nu \rangle \times \ex[\tau_k | Y_k=Y] \allowdisplaybreaks
	 \end{align*}
where (a) follows because $\langle \q_{\| \calk}(T_k), \q_{\perp \calk}(T_k) \rangle = 0$ and $\langle \q_{\| \calk}(T_k), \q_{\perp \calk}(T_k +1) \rangle \leq 0$ by Lemma \ref{lemma: projection_cone} part (i) and (ii), (b) follows as projection to $\calk$ and unused service are both non-negative vectors, (c) follows by Lemma \ref{lemma: projection_cone} part (iii) as $\s(t),\boldsymbol \nu \in \mathcal F$ for any $t$ and (d) follows because the event $\{\tau_k = m\}$ does not depend on $\{\ar(T_k +i)\}_{i\geq m}$, and so we can use the general form of Wald's equation.

	 Then by using Lemma \ref{lemma: drift_results} part (iv),
	 \begin{align*}
    \mathbb{E} \big[ & \Delta W_{\perp \mathcal{K}}(Y) \big| Y_k = Y \big] \allowdisplaybreaks\\
    &\leq \frac{K}{2\|\q_{\perp \mathcal{K}}\|} - \frac{\epsilon}{\|\q_{\perp \mathcal{K}}\|}\langle \q_{\perp \calk}, \boldsymbol \nu \rangle \times \ex[\tau_k | Y_k =Y] - \nu_{\min} \allowdisplaybreaks \\
    & \cblue \leq \frac{K}{2\|\q_{\perp \mathcal{K}}\|} + \frac{\epsilon}{\|\q_{\perp \mathcal{K}}\|}|\langle \q_{\perp \calk}, \boldsymbol \nu \rangle| \times \ex[\tau_k | Y_k =Y] - \nu_{\min} \allowdisplaybreaks \\
    &\stackrel{(a)}\leq \frac{K}{2\|\q_{\perp \mathcal{K}}\|} + \epsilon \norm{\boldsymbol \nu} \times \ex[\tau_k | Y_k =Y] - \nu_{\min} \allowdisplaybreaks\\
    &\stackrel{(b)} \leq \frac{K}{2\|\q_{\perp \mathcal{K}}\|} + \frac{\epsilon}{\delta} \norm{\boldsymbol \nu} - \nu_{\min} \allowdisplaybreaks\\
    &\stackrel{(c)} \leq - \frac{\nu_{\min}}{4}, 
\end{align*}
where (a) follows by the Cauchy-Schwarz inequality, \cblue (b) follows by Lemma \ref{prop:tau_dominate} \cblack and (c) follows whenever $\epsilon \leq \epsilon_0 = \delta \nu_{\min} /2\norm{\boldsymbol \nu}$ and  $\|\q_{\perp \mathcal{K}}\| \geq 2K/\nu_{\min}$. This fulfils the condition C.1 with $\kappa  = 2K/\nu_{\min}$ and $\eta = \nu_{\min}/4$. Then, by using Lemma \ref{lemma: moment_bound} part (ii), the MGF of  $  W_{\perp \calk}(\overline Y) $ exists in a neighbourhood around 0.

Now suppose that time instant $t$ lies between the stopping times $T_k$ and $T_{k+1}$, then, 
\begin{align*}
	\norm{\q_{\perp \calk } (t)} &\leq \norm{\q_{\perp \calk } (T_k)} + na_{\max}(t-T_k)\\
	& \leq \norm{\q_{\perp \calk } (T_k)} + na_{\max}(T_{k+1}-T_k)\\
	& = W_{\perp \calk}( Y_k) + na_{\max}\tau_k.
\end{align*}
So, $\norm{\q_{\perp \calk } (t)}$ is stochastically dominated by $W_{\perp \calk}( Y_k) + na_{\max}\tau_k$, which in turn is stochastically dominated by $W_{\perp \calk}( Y_k) + na_{\max}M$ by Lemma \ref{prop:tau_dominate}. Then in steady state, $\norm{\bar \q_{\perp \calk } }$ is stochastically dominated by $W_{\perp \calk}( \overline Y) + na_{\max}M$. This implies that MGF of $\norm{\bar \q_{\perp \calk } }$ exists in a neighbourhood around 0 and so there exists $C_r<\infty$ such that $\ex \big [ \norm{\bar \q_{\perp \calk } }^r \big] < C_r$ for any $r \in \{1,2,\dots\}$. This establishes SSC and thus by Theorem \ref{prop: ssc}, the heavy traffic scaled mean sum queue length for any algorithm in class $\Pi_3$ satisfies Eq. \eqref{eq: heavy_traffic_sum}. This completes the proof.
\end{proof}

\subsection{Appendix K}

\begin{proof}[Proof of Proposition \ref{prop: beta_algos}]
	\cblue
	\begin{itemize}
	
	    \item[(i)] From the proof of Theorem \ref{thm: ssc_class_1} given in Appendix D, we know any scheduling algorithm in class $\Pi_1(\boldsymbol \nu)$ satisfies the condition C.1 and C.2 with parameters, $\kappa = 2K/W_1$, $\eta = W_1/4$ and $D = n a_{\max}$, where $K = \norm{\boldsymbol \lambda}^2 + \norm{\boldsymbol \sigma}^2 +n$. For uniform Bernoulli traffic, we know that $\norm{\boldsymbol \sigma}^2$ is $O(n)$ which gives us that $K$ is $O(n)$. Taking $W_1$ to be $O(n^{-\alpha_1})$, $\kappa$ is $O(n^{1+\alpha_1})$, $\eta$ is $O(n^{-\alpha_1})$ and $D$ is $O(n)$. Thus, by part (i) of Lemma \ref{lemma: moment_bound}, $\mathbb{E} \Big [\norm{\Bar{\q}_{\perp \mathcal{K}}}^r \Big]$ is $O(n^{r(2+\alpha_1)})$. Then by Lemma \ref{thm: beta_result}, we need $\beta > 3+ \alpha_1$.

	    For power-of-$d$ and random $d$-flip, we have that $W_1 = 1/2n^3$ and so $\alpha_1 = $
	    This gives us that power-of-$d$ and random $d$-flip satisfies Eq. \eqref{eq: beta_result} for $\beta >6$.
	
	
    \item[(ii)]	From the proof of Theorem \ref{thm: ssc_class_2} given in Appendix G, we know any scheduling algorithm in class $\Pi_2$ satisfies the condition C.1 and C.2 with parameters, $\kappa = 2K'/\nu_{\min}$,
	  $\eta = \nu_{\min}/4$ and $D = n a_{\max}$, where $K' = \norm{\boldsymbol \lambda}^2 + \norm{\boldsymbol \sigma}^2 +n + W_2$. By taking $W_2$ to be $O(n^{\alpha_2})$, $\kappa$ is $O(n^{\max\{2,1+\alpha_2\}})$, $\eta$ is $O(n^{-1})$ and $D$ is $O(n)$. Thus, by part (i) of Lemma \ref{lemma: moment_bound}, $\mathbb{E} \Big [\norm{\Bar{\q}_{\perp \mathcal{K}}}^r \Big]$ is $O(n^{r(1+\max\{2,\alpha_2\}})$. Then by Lemma \ref{thm: beta_result}, we need $\beta > 2 + \max\{2,\alpha_2\}$.  
	  
	  For bursty MaxWeight and pipelined MaxWeight, we know that $W_2 = 2mna_{\max}$. Suppose $m = O(n^\gamma)$. Then, for bursty MaxWeight and pipelined MaxWeight, $W_2$ is $O(n^{1+\gamma})$ and so we need $\beta > 3 + \max\{1,\gamma\}$.
	  
	  \color{black}
\end{itemize}

\end{proof}

\subsection{Appendix L}
\label{appendix: large_scale}

\begin{proof}[Proof of Theorem \ref{thm: beta_result}]

	From \cite[Theorem 2]{doi:10.1287/15-SSY193}, we know that 
	\begin{align*}
		-B_1(\epsilon,n) \leq \ex \Big[ \sum_{ij}\bar q_{ij} \Big] - \frac{1}{\epsilon} \Big( n-\frac{3}{2} + \frac{1}{2n} \Big) \leq B_2(\epsilon,n),
	\end{align*}
	where 
	\begin{align*}
		B_1(\epsilon,n) &= \Big( 1- \frac{\epsilon}{2}\Big)\Big( n-2 + \frac{1}{n} \Big) +n - \frac{1}{2} \\
		&\quad + 3n^{(2-\frac{1}{r})} \epsilon^{-\frac{1}{r}}C_r^{\frac{1}{r}}\\
		B_2(\epsilon,n) &= - \Big( 1- \frac{\epsilon}{2}\Big)\Big( n-2 + \frac{1}{n} \Big) +\frac{n+1}{2} \\
		&\quad + 2n^{(2-\frac{1}{r})} \epsilon^{-\frac{1}{r}}C_r^{\frac{1}{r}}.
	\end{align*} 
	By using that $\epsilon(n) = \Theta(n^{-\beta})$, we have that 
	\begin{equation*}
		\frac{1}{\epsilon(n)} \Big( n-\frac{3}{2} + \frac{1}{2n} \Big) = \Theta(n^{1+\beta})
	\end{equation*}
	Now, if $C_r = O(n^{r\alpha})$ for $\alpha>0$, we get that
	\begin{align*}
		n^{(2-\frac{1}{r})} \epsilon^{-\frac{1}{r}}C_r^{\frac{1}{r}} = O(n^{2+\alpha + \frac{\beta - 1}{r}}).
	\end{align*}
	This gives us that, for any $r \geq 1$, we have that, 
	\begin{align*}
		B_1(\epsilon(n),n) &= O(n^{2+\alpha + \frac{\beta - 1}{r}})\\
		B_2(\epsilon(n),n) &= O(n^{2+\alpha + \frac{\beta - 1}{r}})
	\end{align*}
	Then, $\ex \Big[ \sum_{ij}\bar q_{ij} \Big] $ is $\Theta(n^{1+\beta})$ if for some $r \geq 1$,
	\begin{equation*}
		2+\alpha + \frac{\beta - 1}{r} < 1+\beta
	\end{equation*}
	Note that the above equation is satisfied for $r$ large enough if $\beta > 1+ \alpha$.
\end{proof}

\end{techreport}

\end{document}